\documentclass[a4paper, 10pt, reqno, dvipsnames]{amsart}

\usepackage[utf8]{inputenc}

\usepackage{microtype}

\usepackage[english]{babel}

\usepackage{adjustbox}

\usepackage{verbatim}
\usepackage{newcent}

\usepackage{xcolor}

\usepackage{longtable}
\usepackage{graphicx, amsmath, amsfonts, amssymb, amsthm, amscd, amsbsy, amstext, tikz,tikz-cd, mathrsfs, mathtools, enumerate, enumitem, tensor, a4wide,xfrac}   
\usetikzlibrary{decorations.pathreplacing}
\usetikzlibrary{calc}

\usepackage{latexsym,ifthen,stmaryrd,fancyhdr,empheq,verbatim, epigraph}

\newcommand{\changelocaltocdepth}[1]{%
	\addtocontents{toc}{\protect\setcounter{tocdepth}{#1}}%
	\setcounter{tocdepth}{#1}%
}

\makeatletter 
\def\l@subsection{\@tocline{2}{0pt}{2pc}{6pc}{}} \makeatother

\usepackage{capt-of}	

\usepackage{dynkin-diagrams}

\usepackage[colorlinks=true, allcolors=black]{hyperref}

\usepackage{cleveref} 

 \usepackage{mathabx}
 \usepackage{bbm}
 \usepackage{pifont}
 \usepackage{manfnt}

\usepackage{tikz-cd}
\usepackage{mathtools}

\makeatletter
\@namedef{subjclassname@2020}{%
  \textup{2020} Mathematics Subject Classification}
\makeatother

\usepackage[OT2, T1]{fontenc} 

\usepackage[textsize=scriptsize]{todonotes}

\usepackage{scalerel, stackengine}

\usepackage{graphicx}

\usepackage{todonotes} 

\mathtoolsset{
  showmanualtags}    

\hypersetup{anchorcolor=black, linktocpage=false,
colorlinks=true,
citecolor=RawSienna,
linkcolor=Mahogany,
urlcolor=black,
breaklinks=true,
plainpages=true
}

\usetikzlibrary{positioning,shapes,shadows,arrows}
\usepackage[mathscr]{euscript} 
\allowdisplaybreaks

\DeclareRobustCommand{\SkipTocEntry}[5]{}

\usepackage[all]{xy}
\usepackage{scalerel}
\usepackage{graphicx}
\newcommand{\vsimeq}{\scaleobj{1}{\rotatebox[origin=c]{-90}{$\simeq$}}} 

\usepackage{calligra}

\DeclareMathAlphabet{\mathcalligra}{T1}{calligra}{m}{n}
\DeclareFontShape{T1}{calligra}{m}{n}{<->s*[2.2]callig15}{}

\newtheorem{theorem}{Theorem}
\newtheorem{prop}[theorem]{Proposition}
\newtheorem{lemma}[theorem]{Lemma}
\newtheorem{coro}[theorem]{Corollary}

\newtheorem*{corollary*}{Corollary}
\newtheorem*{theorem*}{Theorem}
\newtheorem*{proposition*}{Proposition}
\newtheorem*{conjecture*}{Conjecture}
\numberwithin{equation}{section}
\numberwithin{theorem}{section}

\theoremstyle{remark}
\newtheorem{rmk}[theorem]{Remark}
\newtheorem{notation}[theorem]{Notation}

\theoremstyle{definition}

\newtheorem*{defi*}{Definition}

\newtheorem{deff}[theorem]{Definition}

\newcommand{\cC}{{\mathcal C}}
\newcommand{\cD}{{\mathcal D}}
\newcommand{\cE}{{\mathcal E}}
\newcommand{\cF}{{\mathcal F}}

\newcommand{\cI}{{\mathcal I}}

\newcommand{\cL}{{\mathcal L}}

\newcommand{\cN}{{\mathcal N}}
\newcommand{\cO}{{\mathcal O}}
\newcommand{\cP}{{\mathcal P}}
\newcommand{\cQ}{{\mathcal Q}}

\newcommand{\cS}{{\mathcal S}}
\newcommand{\cT}{{\mathcal T}}

\newcommand{\cV}{{\mathcal V}}

\newcommand{\cX}{{\mathcal X}}

\newcommand{{\bull}}{{\scriptscriptstyle{\bullet}}}

\newcommand{\Left}{\mathsf{Left}}

\newcommand{\Leftbeta}{\mathsf{Left}^{\mathsf{str}}}

\newcommand{\Leftlax}{\mathsf{Left}^{\mathsf{opd}}}

\newcommand{\st}{\mathsf{St}}
\newcommand{\unst}{\mathsf{Unst}}

\newcommand{\Op}{\mathsf{Op}}

\newcommand{\Cat}{\mathsf{Cat}}

\newcommand{\alg}{\mathsf{Alg}}
\newcommand{\env}{\mathsf{Env}}

\newcommand{\omg}{\ensuremath{\mathbf{\Omega}}}

\newcommand{\id}{\text{id}}

\newcommand{{\op}}{{{\rm op}}}
\newcommand{{\coop}}{{{\rm coop}}}

\usepackage{pict2e}

\makeatletter
\newcommand{\adjunction}[4]{%
  #1\colon #2%
  \mathrel{\vcenter{%
    \offinterlineskip\m@th
    \ialign{%
      \hfil$##$\hfil\cr
      \longrightharpoonup\cr
      \noalign{\kern-.3ex}
      \smallbot\cr
      \longleftharpoondown\cr
    }%
  }}%
  #3 \noloc #4%
}
\newcommand{\longrightharpoonup}{\relbar\joinrel\rightharpoonup}
\newcommand{\longleftharpoondown}{\leftharpoondown\joinrel\relbar}
\newcommand\noloc{%
  \nobreak
  \mspace{6mu plus 1mu}
  {:}
  \nonscript\mkern-\thinmuskip
  \mathpunct{}
  \mspace{2mu}
}
\newcommand{\smallbot}{%
  \begingroup\setlength\unitlength{.15em}%
  \begin{picture}(1,1)
  \roundcap
  \polyline(0,0)(1,0)
  \polyline(0.5,0)(0.5,1)
  \end{picture}%
  \endgroup
}
\makeatother

\begin{document}

\title{A straightening-unstraightening equivalence for $\infty$-operads} 

\author{Francesca Pratali}

\begin{abstract}
We provide a straightening-unstraightening adjunction for $\infty$-operads in Lurie's formalism, and show it establishes an equivalence between the $\infty$-category of operadic left fibrations over an $\infty$-operad $\cO^\otimes$ and the $\infty$-category of $\cO^\otimes$-algebras in spaces. In order to do so, we prove that the Hinich-Moerdijk comparison functors induce an equivalence between the $\infty$-categories of operadic left fibrations and dendroidal left fibrations over an $\infty$-operad, and we characterize, for any symmetric monoidal $\infty$-category $\cC^\otimes$, the essential image of the monoidal unstraightening functor restricted to strong monoidal functors $\cC^\otimes\to \cS^\times$. 
 \end{abstract}

\address{F.P.: LAGA, Universit\'e Sorbonne Paris Nord, 99 avenue Jean-Baptiste Cl\'ement, 93430 Villetaneuse, France}

\email{pratali@math.univ-paris13.fr}



\maketitle

\tableofcontents

\section*{Introduction and main results}
\subsection*{Introduction} The theory of operads offers a general framework for encoding algebraic structures. Initially introduced in the early 1970s to capture homotopy-coherent algebraic operations on topological spaces (\cite{M:TGILS}, \cite{BV:HIASTS}), it has since then found numerous applications in different fields of mathematics besides algebraic topology, such as algebra, mathematical physics, algebraic geometry and combinatorics.

\noindent In the present work, we regard an operad $P$ as a multicategory: informally, an operad consists of a set of objects $P_0$ and, for any choice of objects $x_1,\dots,x_n,x $ in $ P_0$, multi-hom sets $P(x_1,\dots,x_n; x)$ together with the specification of composition laws. Set-valued functors on a category $C$ generalize to set-valued algebras on an operad $P$, which consist in a collection of sets $\{F(x)\}_{x\in P_0}$ together with action maps $ F(x_1)\times \dots \times F(x_n) \to F(x)$ for every multimorphism $(x_1,\dots,x_n)\to x$ in $P$, compatible with operadic composition.  For a symmetric monoidal category $V$ and $x_1,\dots,x_n,x$ objects in it, one can consider a morphism $x_1\otimes \dots \otimes x_n \to x$ in $V$ as a `representable' multimorphism $(x_1,\dots,x_n)\to x$, so that $V$ naturally yields an operad $\underline{V}$ with the same set of objects of $V$ and multi-hom sets given by $\underline{V}(x_1,\dots,x_n;x)= V(x_1\otimes\dots\otimes x_n;x)$; set-valued $\underline{V}$-algebras coincide with \emph{lax} symmetric monoidal functors $(V,\otimes)\to (\mathsf{Set},\times)$.

\noindent From a modern perspective, the notion of operad can be replaced by its fully homotopy-coherent analogue, known as \emph{$\infty$-operad}. One can say that $\infty$-operads offer a framework to work with homotopy-coherent algebraic structures in the setting of $\infty$-categories, the idea being that an $\infty$-operad $\cP$ has now multi-hom  \emph{spaces} $\cP(x_1,\dots,x_n;x)$ for any choice of objects $x_1,\dots,x_n,x$, and a $\cP$-algebra is a collection of spaces $F(x)\in \cS$ for any color $x$ of $\cP$, together with appropriately encoded action maps $\cP(x_1,\dots , x_n; x)\times F(x_1)\times \dots \times F(x_n)\to F(x)$. As it happens in the strict case, the theory of $\infty$-operads and algebras over them generalizes both the theory of $\infty$-categories and that of symmetric monoidal $\infty$-categories and lax monoidal functors.

\noindent Just as in the case of $\infty$-categories, there are several, equivalent models for $\infty$-operads, including those of Lurie \cite{Lu:HA}, Barwick \cite{B:FOCHO}, Moerdijk–Weiss dendroidal sets \cite{MW:DS} and Cisinki–Moerdijk complete dendroidal Segal spaces \cite{CM:DSSIO}. Although all of these models for $\infty$-operads result in equivalent $\infty$-categories (\cite{HeuHiMoe:OEBLMDMIO}, \cite{B:FOCHO}, \cite{CHH:TMHTIO}, \cite{HM:OELIODIO}), each of the model has its own dis/advantages. In this article, we use Lurie $\infty$-operads and dendroidal $\infty$-operads in the form of complete dendroidal Segal space. In Lurie's formalism, an $\infty$-operad is represented by its $\infty$-category of operators, and homotopy-coherence of the constructions is automatically ensured, while the operadic intuition is sometimes hard to recover. On the other hand, the dendroidal formalism represents an $\infty$-operad as a functor $\omg^{\text{op}}\to \cS$, where $\omg$ is a category of trees giving the shapes of multimorphisms and their compositions, so the combinatorics are often more involved or of a different nature than the ones for simplicial sets. 

\noindent In the context of $\infty$-categories, the Lurie-Grothendieck equivalence \cite{Lu:HTT}, also known as straightening-unstraightening equivalence, gives a way of manipulating $\cS$-valued functors out of an $\infty$-category $\cC$. This is accomplished by establishing an equivalence 
$$\adjunction{\st^\cC}{\Left_\cC}{\mathsf{Fun}(\cC, \cS)}{\unst^\cC}$$ between the $\infty$-category of functors $\cC\to \cS$ and the $\infty$-category $\Left_\cC$ of left fibrations $\cD\to \cC$ over $\cC$, the higher categorical generalisation of categories fibred in groupoids in ordinary category theory. The un/straightening theorem generalizes Grothendieck’s equivalence between categories cofibred in groupoids and set-valued pseudofunctors, as well as the correspondence between covering spaces and sets with an action of the fundamental group of the base, and is a crucial tool in the theory of symmetric monoidal $\infty$-categories and the theory of $\infty$-operads.

\noindent When $\cC$ is the underlying $\infty$-category of a symmetric monoidal $\infty$-category $\cC^\otimes$, the un/straightening equivalence allows us to similarly work with \emph{lax monoidal} functors $\cC^\otimes\to \cS^\times$. This was first proven by Hinich in \cite{H:RAM}, who proved that the un/straightening equivalence is the underlying functor of an equivalence $$\adjunction{\st^{\cC,\otimes}}{\mathsf{sm}\Left_{\cC^\otimes}}{\mathsf{Fun}^{\text{lax}}(\cC^\otimes, \cS^\times)}{\unst^{\cC,\otimes}}, $$ where $\mathsf{sm}\Left_{\cC^\otimes}$ is the $\infty$-categories of left fibrations over $\cC^\otimes$ internal to symmetric monoidal $\infty$-categories. Motivated by Ramzi's result in \cite{R:MGCIC}, of which the above one is a corollary, in the context of this work we refer to the above equivalence the \emph{monoidal un/straightening equivalence}.

\noindent Following the same lines, a straightening-unstraightening equivalence for $\infty$-operads needs to establish an equivalence of $\infty$-categories $$ \adjunction{\st^\cP}{{\mathsf{Left}\left(\Op_\infty\right)}_{\cP}}{\alg_\cP(\cS^\times)}{\unst^\cP}$$ between an $\infty$-category of left fibrations over $\cP$ internal to $\infty$-operads and the $\infty$-category of $\cS$-valued $\cP$ algebras.

\noindent In this article, we construct and prove a straightening-unstraightening equivalence for Lurie $\infty$-operads.

\subsection*{Strategy and main results} Our construction makes essential use of two fundamental tools: the \emph{Hinich-Moerdijk comparison functors} (\cite{HM:OELIODIO}), which consist of a pair of adjoint functors that establish a direct equivalence $$ \adjunction{\delta}{\ell\Op_\infty}{\mathsf{D}\Op_\infty}{\lambda}$$ between the $\infty$-category of Lurie $\infty$-operads and that of dendroidal $\infty$-operads, and the \emph{symmetric monoidal envelope functor}, a functor $$  \env(-)^\otimes\colon \ell\Op_\infty\longrightarrow \mathsf{sm}\Cat_\infty$$ from the $\infty$-category of Lurie $\infty$-operads to that of symmetric monoidal $\infty$-categories with strong monoidal functors, left adjoint to the forgetful functor. Its universal property provides, for any Lurie $\infty$-operad $\cO^\otimes$ and any symmetric monoidal $\infty$-category $\cC^\otimes$, a natural equivalence $$\alg_{\cO^\otimes}(\cC^\otimes)\simeq \mathsf{Fun}^{\text{str}}(\env(\cO)^\otimes, \cC^\otimes)$$ between the $\infty$-category of $\cO^\otimes$-algebras in $\cC^\otimes$ and strong monoidal functors $\env(\cO)^\otimes \to \cC^\otimes$. 

\noindent First, we define the $\infty$-category $\Leftlax_{\cO^\otimes}\coloneqq \Left(\ell\Op_\infty)_{\cO^\otimes}$ of \emph{operadic left fibrations}  over a Lurie $\infty$-operad $\cO^\otimes$ (\Cref{opdleft}): it is the full sub $\infty$-category of ${\ell\Op_\infty}_{/\cO^\otimes}$ spanned by those objects $(\cX^\otimes, p)$ for which $p\colon \cX^\otimes \to \cO^\otimes$ is a left fibration of $\infty$-categories.

\noindent Our first result consists in proving that our definition is consistent with an analogous notion of left fibration formulated in the dendroidal formalism. We consider the notion of \emph{dendroidal left fibration} of dendroidal spaces as defined by Boavida-Moerdijk in \cite{BdBM:DSGSSBPQT}, which takes its origins in Heuts definition \cite{He:AOIO} for dendroidal sets. 

\noindent More precisely, consider a dendroidal $\infty$-operad $X$ and write $\mathsf{D}\Left_X$ for the $\infty$-category of dendroidal left fibrations over $X$. It is a full sub $\infty$-category of the over $\infty$-category ${\mathsf{DOp}_\infty}_{/X}$.
\begin{theorem*}[\Cref{luppolo}]
	For any Lurie $\infty$-operad $\cO^\otimes$, the Hinich-Moerdijk comparison functors induce an equivalence of $\infty$-categories $$ \adjunction{\delta}{\Leftlax_{\cO^\otimes}}{\mathsf{D}\Left_X}{\lambda},$$ where $X\simeq \delta (\cO^\otimes)$ is the dendroidal model for $\cO^\otimes$.
\end{theorem*}

\noindent A consequence of this result is \Cref{fibrewise1}, where we characterize equivalences between operadic left fibrations as fibrewise homotopy equivalences of spaces. 

\noindent As a second step, motivated by the universal property of the envelope, we characterize the left fibrations $(\cD^\otimes, \alpha^\otimes)$ in $\mathsf{sm}\Left_{\cC^\otimes}$ that correspond to the monoidal unstraightening of strong monoidal functors for \emph{any} symmetric monoidal $\infty$-category $\cC^\otimes$. 

\begin{proposition*}[\Cref{tiprego}]
Let $(\cD^\otimes, \alpha^\otimes)$ be an object of $\mathsf{sm}\Left_{\cC^\otimes}$. 
The following conditions are equivalent: 
\begin{enumerate}
	\item  $(\cD^\otimes, \alpha^\otimes)\simeq \unst^{\cC,\otimes}(F^\otimes)$ for a \emph{strong} monoidal functor $F^\otimes \colon \cC^\otimes 	\to \cS^\times$;
	\item for any $x\in \cC^\otimes_{\underline{n}_+}$, $x\simeq x_1\oplus\dots\oplus x_n $, the induced map between the fibres $$ (\cD^\otimes,\alpha^\otimes)_{x_1\oplus\dots\oplus x_n}\longrightarrow (\cD^\otimes,\alpha^\otimes)_{x_1\otimes\dots\otimes x_n}$$ induced by $\beta_! \colon x\to x_1\otimes \dots \otimes x_n$ in $\cC^\otimes$ is an equivalence.
\end{enumerate}
\end{proposition*}

\noindent We call $\mathsf{sm}\Leftbeta_{\cC^\otimes}$ the full subcategory of $\mathsf{sm}\Left_{\cC^\otimes}$ spanned by those objects satisfying the condition in \Cref{tiprego}; as it is the essential image of strong monoidal functors, we call the elements therein \emph{strong symmetric monoidal (sm) left fibrations} over $\cC^\otimes$. We use \Cref{tiprego} with $\cC^\otimes \simeq \env(\cO)^\otimes$, and by universal property of the envelope we obtain the equivalence
$$\adjunction{\st^{\env(\cO),\otimes}}{\mathsf{sm}\Leftbeta_{\env(\cO)^\otimes}}{\alg_{\cO^\otimes}(\cS^\times)}{\unst^{\env(\cO),\otimes}}.$$

\noindent The next step is to relate the $\infty$-category of strong sm left fibrations over $\env(\cO)^	\otimes$ with that of operadic left fibrations over  $\cO^\otimes$. In order to do so, we observe that, albeit the envelope is not a fully faithful functor, it is a consequence of \cite[Proposition 2.4.3]{HK:IOSMIC} that it becomes so when sliced over the terminal $\infty$-operad. In particular, for any Lurie $\infty$-operad $\cO^\otimes$ the envelope defines a fully faithful left adjoint
$$
\adjunction{\env(-)^\otimes}{{\ell\Op_\infty}_{/\cO^\otimes}}{{\mathsf{sm}\Cat_\infty}_{/\env(\cO^\otimes)}}{G'}, $$
\noindent where the right adjoint $G'$ is the base-change along the unit $\cO^\otimes \to \env(\cO)^\otimes$ after the forgetful functor. In \Cref{kriku} we prove that this adjunction restricts to a well-defined adjunction $$ \adjunction{\env(-)^\otimes}{\Leftlax_{\cO^\otimes}}{\mathsf{sm}\Left_{\env(\cO)^\otimes}}{G}$$ where, by applying a consequence of \Cref{luppolo}, the right adjoint is conservative, which means that the adjunction is in fact an equivalence.

\noindent These results together allow to prove our straightening-unstraightening equivalence for $\infty$-operads. 
\begin{theorem*}[\Cref{mainnnn}]
	For any Lurie $\infty$-operad $\cO^\otimes$, there is an equivalence of $\infty$-categories $$ \adjunction{\st^\cO}{\Leftlax_{\cO^\otimes}}{\alg_{\cO^\otimes}(\cS^\times)}{\unst^\cO},$$ where the left adjoint is given by the composition $$ \st^{\cO}\colon \Leftlax_{\cO^\otimes} \xrightarrow{\env(-)^\otimes}\mathsf{sm}\Leftbeta_{\env(\cO)^\otimes} \xrightarrow{\st^{\env(-),\otimes}} \mathsf{Fun}^{\text{str}}(\env(\cO)^\otimes,\cS^\times)\simeq \alg_{\cO^\otimes}(\cS^\times).$$ 
	
\end{theorem*}

\noindent Exploiting the presentation of the monoidal straightening equivalence for discrete categories of \cite[Theorem 4.4]{P:RDLF}, we can also state the following explicit formula for the operadic straightening functor when $\cO^\otimes$ is discrete.

\begin{corollary*}[\Cref{coromoro}]
	For a \emph{discrete} $\infty$-operad $\cO^\otimes$, the operadic straightening equivalence can be explicitely written in the following terms. Given an operadic left fibration $(\cT^\otimes, \alpha^\otimes)$ over $\cO^\otimes$ and an object $x$ of $\cO$, the value at $x$ of the $\cO^\otimes$-algebra $\st^{\cO}(\cT^\otimes, \alpha^\otimes)$ is given by $$ \st^\cO(\cT^\otimes, \alpha^\otimes)(x)\simeq \env(\cT)\times_{\env(\cO)} {\env(\cO)}_{/x}.$$ 
\end{corollary*}

\subsection*{Relation with other works }\label{relation} Let us make some comments about the relation with other straightening-unstraightening equivalences.

\begin{enumerate}
	
	\item  An operadic un/straightening equivalence $\Leftlax_{\cO^\otimes}\simeq \alg_{\cO^\otimes}(\cS^\times)$ can be deduced from Ramzi's \emph{$\cO$-monoidal Grothendieck construction} of \cite{R:MGCIC}, and in particular from [Corollary C] therein, which, by chosing $\cC=\cO$ yields an equivalence of $\infty$-categories  
	\begin{equation}\label{maxime}
		\mathsf{Left}_{\cO}^{\cO-lax}\simeq \mathsf{Fun}^{\cO-lax}(\cO,\cS), 
	\end{equation}  and one easily sees that $\mathsf{Left}_{\cO}^{\cO-lax}\simeq \mathsf{Left}^\mathsf{opd}_{\cO^\otimes}$, see also \Cref{crefs}.
	
	\noindent The same [Corollary C], when instantiated for $\cO=\mathsf{Fin}_*$ and $\cC$ a symmetric monoidal $\infty$-category, implies the equivalence of $\infty$-categories $\mathsf{smLeft}_{\cC^\otimes}\simeq \mathsf{Fun}^{\mathsf{lax}}(\cC^\otimes,\cS^\times)$ which we use in \Cref{tiprego} and ultimately to prove \Cref{mainnnn}. The latter equivalence was already proven by Hinich in \cite{H:RAM} without relying on the monoidal Grothedieck construction, and this makes our work independent from \cite{R:MGCIC}.
	
	\noindent The equivalence between the $\infty$-categories of dendroidal and operadic left fibrations of \Cref{luppolo} allows to characterize the equivalences in $\mathsf{Left}^\mathsf{opd}_{\cO^\otimes}\simeq \mathsf{Left}_\cO^{\cO-lax}$ as the fibrewise equivalences, an insight which is of independent interest.
	
	\noindent Additionally, in \Cref{tiprego}, we describe the unstraightening of \emph{strong} monoidal functors, which, for us, cannot be immediately inferred from \cite{R:MGCIC}.

	\item In the dendroidal formalism, dendroidal un/straightening equivalences are constructed in the works of Heuts \cite{He:AOIO}, for dendroidal sets, and Boavida-Moerdijk \cite{BdBM:DSGSSBPQT}, for complete dendroidal Segal spaces. The equivalences are realized as (zig-zags of) Quillen equivalences between the \emph{covariant model structure} for dendroidal left fibrations and the projective model structure on simplicial algebras over (the Boardman-Vogt resolution of) a dendroidal $\infty$-operad.
	
	\noindent In contrast to this, we construct the operadic un/straightening functors as the composition of functors of $\infty$-categories, and using the envelope allows to reduce the discussion to symmetric monoidal $\infty$-categories. 
	
	\noindent The equivalence proven in \Cref{luppolo} between dendroidal and operadic left fibrations suggests looking for presentations of our operadic un/straightening equivalence as Quillen equivalences of model categories, either with the dendroidal formalism (see \cite{P:RDLF}) or with Lurie's formalism.
	
	\noindent Finally, we observe that by combining our operadic un/straightening equivalence with the dendroidal one we obtain an equivalence between the $\infty$-categories of algebras over a dendroidal $\infty$-operad and that of algebras for its Lurie equivalent.
	
	\end{enumerate}

\noindent The followings are remarks of a more speculative nature.

\begin{enumerate}
\setcounter{enumi}{2}

\item In \cite{K:MEGCDSO}, Kern proposes a construction of the symmetric monoidal envelope of dendroidal $\infty$-operads, uses it to define coCartesian fibrations of such and to prove a un/straightening equivalence for dendroidal $\infty$-operads. We can ask whether our strong monoidality condition for $\mathsf{sm}\Leftbeta_{cC^\otimes}$ corresponds to that of being \emph{equifibred} formulated in \cite{K:MEGCDSO}, in which case it would be reasonable to expect that Hinich-Moerdijk comparison functors induce an equivalence between Kern's dendroidal symmetric monoidal envelope and the one for Lurie $\infty$-operads. 
\item The fact that the symmetric monoidal envelope functor is fully faithful on slices has been first observed by Barkan, Haugseng and Steinebrunner in \cite{BHS:EAP}. In the same article, they characterize the essential image of $\env(-)^\otimes \colon \ell\Op_\infty\to{\mathsf{sm}\Cat_\infty}_{/\env(\mathsf{Com})^\otimes}$, but we do not make use of this result in our work. It would be interesting to compare our un/straightening of \Cref{mainnnn} with \cite[Theorem E]{BHS:EAP} and with \cite[Theorem 2.3.5.]{BHS:EAP}. In this respect, it could also be interesting to consider the work by Haugseng, resp. Haugseng and Kock, in \cite{H:IOVSS}, resp. \cite{HK:IOSMIC}.
\end{enumerate}

\subsection*{Outline} In \Cref{section2}, we set up notation and $\infty$-categorical conventions, recall Lurie and the dendroidal formalism for $\infty$-operads and the notion of left fibration of $\infty$-categories.

\noindent We devote \Cref{section3} to the definition of operadic left fibrations and the study of the $\infty$-category of such. In particular, in \Cref{SS:1:1} we define the notion of \emph{operadic left fibration} for a Lurie $\infty$-operad. In \Cref{SS:1:2}, we recall the definition of dendroidal left fibrations for dendroidal $\infty$-operads. After some recollections on Hinich-Moerdijk comparison functors, in \Cref{SS:1:3} we prove \Cref{luppolo}, that is, that they restrict to an equivalence between the $\infty$-category of dendroidal, resp. operadic, left fibrations, which is a result of independent interest.

\noindent  In  \Cref{SS:2:1}, we recall the monoidal structures on $\mathsf{Fun}(\cC,\cS)$ and $\Left_\cC$ when $\cC$ is the underlying $\infty$-category of a symmetric monoidal $\infty$-category $\cC^\otimes$, and we state the monoidal straightening-unstraightening equivalence. In \Cref{SS:2:2}, we prove \Cref{malaga}, where we characterize the $\infty$-category $\mathsf{sm}\Leftbeta_{\cC^\otimes}$ of strong sm-left fibrations, that is, the unstraightening of strong monoidal functors $\cC^\otimes \to \cS^\times$.

\noindent In \Cref{round1} we recall the definition of symmetric monoidal envelope, and set up some notation. In \Cref{ranianna}, we prove that strong sm-left fibrations are equivalent to operadic left fibrations, and we do this in two steps. First, in \Cref{pollo} we show that the symmetric monoidal envelope defines a fully faithful left adjoint from operadic left fibrations to sm-left fibrations, then in \Cref{kriku} we show that its essential image coincides with strong sm-left fibrations. In this last step, we make essential use of \Cref{luppolo}.

\noindent Finally, in \Cref{triku} we put everything together and prove \Cref{mainnnn}, the straightening-unstraightening equivalence for $\infty$-operads. 

\subsection*{Acknowledgments} 
I would like to thank Gijs Heuts for the insightful discussions from which this work originated. I want to thank Ieke Moerdijk for helpful remarks on \Cref{section3} and Hiro Lee Tanaka for valuable mathematical and expositary commentaries. I thank also Miguel Barata, Hugo Pourcelot, Bruno Ignacio Galvez Araneda, Victor Carmona Sanchez and Victor Saunier for many helpful discussions, and Jordan Levin for helpful remarks on an earlier version of this article. Finally, I would like to thank my advisor, Eric Hoffbeck, for his constant support and interest in my projects, and for carefully reading my (several) drafts.

\noindent Part of this work originated while visiting the University of Utrecht, in 2024, cofinanced by the Eole grant (RFN). My PhD is founded by the European Union’s Horizon 2020 research and innovation programme under the Marie Skłodowska-Curie grant agreement No 945332. 

\changelocaltocdepth{1}
\section{Recollections and conventions}\label{section2}
\subsection{Conventions on $\infty$-categories} We work in the setting of $\infty$-categories as developed in \cite{Lu:HTT}, a concise and accessible account of which can be found, for example, in \cite{G:SCIO}. Let us recall the main constructions and definitions we will use throughout this article, and set up the notation for next sections. 

\begin{itemize}  
	\item Given an $\infty$-category $\cC$ and objects $x,y$ of $\cC$, we write $\mathsf{Map}_\cC(x,y)$, or sometimes just $\mathsf{Map}(x,y)$, for the mapping space of arrows in $\cC$ from $x$ to $y$.
	\item Given two $\infty$-categories $\cC$, $\cD$, there is an $\infty$-category of functors $\mathsf{Fun}(\cC, \cD)$, satisfying the equivalence $\mathsf{Map}(\cX, \mathsf{Fun}(\cC,\cD))\simeq \mathsf{Map}(\cX \times \cC, \cD)$ for any $\infty$-category $\cX$.
	\item We write $\cS$ for the $\infty$-category of spaces, presented by the Kan-Quillen model structure on simplicial sets.
	\item Given an $\infty$-category $\cC$, we write $\mathsf{PSh}(\cC)$ for the $\infty$-category of presheaves on $\cC$, that is, $\mathsf{PSh}(\cC)\coloneqq \mathsf{Fun}(\cC^{\text{op}},\cS)$. We will still write $\mathsf{Fun}(\cC,\cS)$ instead of $\mathsf{PSh}(\cC^{\text{op}})$ to emphasize covariancy in $\cC$.
	\item We write $\Cat_\infty$ for the $\infty$-category of small $\infty$-categories. It can be realized as the full sub $\infty$-category of simplicial spaces $\mathsf{Psh}(\Delta)=\mathsf{Fun}(\Delta^\text{op},\cS)$ spanned complete Segal spaces, see \cite{R:AMHTHT}. 
	\item A functor of $\infty$-categories  $L\colon \cC \to \cD$ is a \emph{left Bousfield localization}, here called just \emph{localization}, if it has a fully faithful right adjoint $\cD \hookrightarrow \cC$. If $\cD\subseteq \cC$ is a full sub $\infty$-category, we say that $\cD$ is a localization of $\cC$ if the inclusion has a left adjoint. 
	\item We will work with slice $\infty$-categories of the form $\cC_{/x}$, or occasionally also $\cC_{x/}$, with $x$ an object of $\cC$. The objects of the \emph{over-$\infty$-category} $\cC_{/x}$ are pairs $(y,\alpha)$, where $y$ is an object of $\cC$ and $\alpha\colon y \to x$ a morphism in $\cC$, a morphism $ (y,\alpha)\to (z,\beta)$ in $\cC_{/x}$ is a $2$-morphism in $\cC$ of the form 
	\[
	\begin{tikzcd}
	y \arrow[rr] \arrow[rd, "\alpha"']& & z \arrow[dl, "\beta"] \\
		& x &
	\end{tikzcd}
	\]
	\noindent and more generally its $n$-morphisms are diagrams $$ \Delta^{n+1}\simeq \Delta^n \star \Delta^0 \longrightarrow \cC$$ 
	\noindent taking the cocone point into $x$. One can define dually the simplices of the under-$\infty$-category $\cC_{x/}$. 
\end{itemize}
\noindent By replacing the object $x\colon \Delta^0\to \cC$ by a more general diagram $f\colon \cD \to \cC$, as for example the one singled out by a $n$-simplex of $\cC$, $f\colon \Delta^n \to \cC$, one can similarly define the under-$\infty$-category $\cC_{f/}$.

\noindent We state the following important
\begin{prop}[{\cite[Lemma 2.3.4.]{HM:OELIODIO}}]
Let $\cC$ be an $\infty$-category and $X$ an object of $\cC$. The Yoneda embedding induces an equivalence of $\infty$-categories $$\mathsf{PSh}(\cC_{/x})\xlongrightarrow{\simeq}\mathsf{PSh}(\cC)_{/x}.$$
\end{prop}

\subsubsection{Adjunctions between slices}\label{adjslice}
Let $\adjunction{F}{\cC}{\cD}{G}$ be a pair of adjoint functors.

\noindent Given an object $y$ of $\cD$, there is an induced pair of adjoint functors between the slice $\infty$-categories $$ \adjunction{\tilde{F}}{{\cC}_{/G(y)}}{{\cD}_{/y}}{G},$$ where the right adjoint consists in applying $G$, while the left adjoint consists in applying $F$ and then postcomposing with the counit, that is, it the composite 
$$\cC_{/G(y)}\longrightarrow \cD_{/FG(y)} \xlongrightarrow{{\epsilon_y}_*} \cD_{/y}  .$$ 

\noindent If $\cC$ has pullbacks, after \cite[Proposition 5.2.5.1]{Lu:HTT} for any object $x$ of $\cC$ there is an induced pair of adjoint functors between the slice $\infty$-categories
$$ \adjunction{F}{\cC_{/x}}{\cD_{/F(x)}}{G'},$$ where the left adjoint consists in applying $F$, while the right adjoint is the base-change along the unit after having applied $G$, that is, it is the composite $$\cD_{/F(x)}\longrightarrow \cC_{/GF(x)} \xlongrightarrow{\eta_x^*} \cC_{/x}  .$$ 

\noindent 
\subsubsection{Left fibrations}
\noindent The main constructions we will work with are variations or generalization of left fibrations of $\infty$-categories. Let us recall how these latter are defined.

\begin{deff}
Let $p\colon \cC \to \cD$ be a functor of $\infty$-categories and $\alpha\colon X \to Y$ a morphism of $\cC$.  We say that $\alpha$ is \emph{$p$-coCartesian}, or that it is \emph{a $p$-coCartesian lift of $\gamma=p(\alpha)$} if the functor $$ \cC_{\alpha/}\longrightarrow \cC_{X/}\underset{\cD_{p(X)/}}{\times} \cD_{p(\alpha)/}$$ \noindent is an equivalence. The functor $p\colon \cC \to \cD$ is a \emph{coCartesian fibration} if every morphism in $\cD$ has a $p$-coCartesian lift.
\noindent A coCartesian fibration $p\colon \cC \to \cD$ is called a \emph{left fibration} if all morphisms in $\cC$ are $p$-coCartesian, or, equivalently, if the fibre $\cC_Y$ is an $\infty$-groupoid for any $Y$ in $\cD$.
\end{deff}

\noindent We write $\Left_\cC$ for the full sub $\infty$-category of ${\Cat_\infty}_{/\cC}$ spanned by left fibrations, and we denote by $(\cD, p)$, where $p\colon \cD \to \cC$, an object therein. A fundamental property of this $\infty$-category is that equivalences are characterized as fibrewise homotopy equivalences of spaces, as stated by the following

\begin{prop}{\cite[Proposition 1.10]{B:SOGC}}
Let $f\colon (\cD,\alpha)\to (\cE,\beta)$ be a morphism between left fibrations over $\cC$. Then $f$ is an equivalence in $\Left_\cC$ if and only if, for any $x\in \cC$, the induced map between the fibres $$ f_x\colon (\cD,\alpha)_x \longrightarrow (\cE,\beta)_x$$ is an equivalence of spaces. 
\end{prop}

\subsection{Lurie  $\infty$-operads} 
Let us recollect Lurie's formalism for $\infty$-operads.
\begin{itemize}
\item Denote by $\mathsf{Fin}_*$ the skeleton of the category of finite pointed sets. Its objects are the finite sets $\underline{n}_+ =\{\ast,1,\dots,n\}$, for $n\geq 0$, where $\ast$ is the base point, and a map $f\colon \underline{n}_+ \to \underline{m}_+$ is a map of sets preserving the base point. The map $f$ is called \emph{active} if $f^{-1}(\ast)=\{\ast\}$, while it is called \emph{inert} if $\#f^{-1}(i)=1$ for all $i\in \underline{m}_+\setminus \{\ast\}$. We will still denote by $\mathsf{Fin}_*$ the simplicial set given by its nerve.
\item For any $n\geq 0$, we denote by $\beta\colon \underline{n}_+\to \underline{1}_+$ the unique active map in $\mathsf{Fin}_*$ from $\underline{n}_+$ to $\underline{1}_+$.
\item For any $n\geq 0$ and $i\in \{1,\dots,n\}$, we denote by $\rho^i\colon \underline{n}_+\to \underline{1}_+$ the inert morphism sending $j\neq i$ to the base point $\ast$ and the element $i$ to $1$. We call $\rho_i$ the \emph{projection} on the $i$-th coordinate, where, when $n=0$, there is a unique morphism $\rho\colon \underline{0}_+\to \underline{1}_+$.
 \end{itemize}

\begin{deff}[{\cite[Definition 2.1.1.10 ]{Lu:HA}}]\label{defi1}
A \emph{Lurie $\infty$-operad} is an object $(\cO^\otimes,p)$ of ${\Cat_{\infty}}_{/\mathsf{Fin}_*}$ such that the morphism $p\colon \cO^\otimes \to \mathsf{Fin_*}$ has the following properties:
\begin{enumerate}
\item any inert morphism $\alpha$ of $\mathsf{Fin}_*$ admits a coCartesian lift $\alpha_!$ in $\cO^\otimes$. As a result, for any inert $\underline{n}_+\to \underline{m}_+$ there is an induced functor of fibres $\cO^\otimes_{\underline{n}_+}\to \cO^\otimes_{\underline{m}_+}$;
\item for every $n\geq 0$, the functor $\cO^\otimes_{\underline{n}_+}\to \prod_{i=1}^n \cO^\otimes_{\underline{1}_+}$ induced by the coCartesian lifts $\{\rho^i_!\}_{i=1}^n$ of the projections is an equivalence of spaces;
\item for any arrow $f\colon \underline{n}_+ \to \underline{m}_+$ and objects $x$ of $\cO^\otimes_{\underline{n}_+}$ and $y$ of $\cO^\otimes_{\underline{m}_+}$, denote by $\mathsf{Map}^f(x,y)$ the subspace of the mapping space $\mathsf{Map}_{\cO^\otimes}(x,y)$ spanned by those maps lying over $f$. Given $f$ and $y$ as above, for any $x$ in $\cO^\otimes_{\underline{n}_+}$, the map $$ \mathsf{Map}^f(x,y) \longrightarrow \prod_{i=1}^m \mathsf{Map}^{f\circ \rho^i}(x,y_i)$$  induced by the projections is an equivalence, where $y \to y_i$ is the coCartesian lift of $\rho^i$.
\end{enumerate}\end{deff}

\noindent We will denote a Lurie $\infty$-operad $(\cO^\otimes, p\colon \cO^\otimes \to \mathsf{Fin_*})$ by $\cO^\otimes$, leaving the map $p$ implicit. 

\noindent Observe that condition $(1)$ implies that $\cO^\otimes$ is an $\infty$-category. Given an $\infty$-operad $\cO^\otimes$, its \emph{underlying category} $\cO$ is the fibre of $\cO^\otimes$ over $\underline{1}_+$, $\cO\simeq \cO^\otimes_{\underline{1}_+}$.

\noindent By condition $(3)$ in \Cref{defi1}, any object $x \in \cO^\otimes_{\underline{n}_+}$ is equivalent to a $n$-uple $(x_1,\dots,x_n) \in \cO^{\times n},$ and for this reason we write $x\simeq x_1\oplus \dots \oplus x_n$. 
\noindent For objects $x_1,\dots,x_n, y$ of $\cO$, the space of multimorphisms $ (x_1,\dots,x_n)\to y$ is spanned by the coCartesian lifts of $\beta$ with domain $x_1\oplus\dots\oplus x_n$.

\noindent Given such an $\infty$-operad, we call \emph{inert} the arrows of $\cO^\otimes$ which are (equivalent to) coCartesian lifts of inert arrows of $\mathsf{Fin}_*$.

\noindent Lurie $\infty$-operads assemble into an $\infty$-category, as we now specify.

\begin{deff}
The $\infty$-category of Lurie $\infty$-operads $\ell\Op_\infty$ is the non full subcategory of ${\Cat_\infty}_{/\mathsf{Fin}_*}$ where 
\begin{itemize}
\item objects are Lurie $\infty$-operads $\cO^\otimes$;
\item a morphism $f\colon \cO^\otimes \to \cD^\otimes$ in ${\Cat_\infty}_{/\mathsf{Fin}_*}$ is in $\ell\Op_\infty$ if $f$ preserves inert arrows. 
\end{itemize}
\end{deff}

\begin{rmk}\label{luriesmodel}
Given a discrete colored operad $P$, we can consider its $1-$category of operators and the natural morphism of this into $\mathsf{Fin}_*$. By considering the nerve of this map one obtains a Lurie $\infty$-operad, denoted by $\cN(P)^\otimes$. In particular, one has $\cN(\mathsf{Com})^\otimes \simeq \mathsf{Fin}_*$.
\end{rmk}

\noindent One can define symmetric monoidal $\infty$-categories as the $\infty$-operads with \emph{representable} multimorphisms. We follow \cite[Definition 2.0.0.7]{Lu:HA}.

\begin{deff}\label{ihp}
A \emph{symmetric monoidal $\infty$-category} $\cC^\otimes$ is a Lurie $\infty$-operad $p\colon \cC^\otimes \to \mathsf{Fin_*}$ where $p$ is furthermore a coCartesian fibration (that is, all morphisms in $\mathsf{Fin}_*$ have coCartesian lifts).
\end{deff}

\noindent The $n$-fold tensor product of $\cC$ corresponds to coCartesian lifts $\beta_!$ of the unique active morphism $\beta\colon \underline{n}_+\to \underline{1}_+$. 

\begin{deff}
The $\infty$-category of symmetric monoidal $\infty$-categories $ \mathsf{smCat}_\infty$ is the non full subcategory of ${\Cat_\infty}_{/\mathsf{Fin}_*}$ where objects are symmetric monoidal $\infty$-categories, and where a morphism $\cC^\otimes \to \cV^\otimes$ over $\mathsf{Fin}_*$ is a morphism in $\mathsf{smCat}_\infty$ if it preserves all the coCartesian lifts.
\end{deff}

\begin{rmk}
The $\infty$-category $\mathsf{smCat}_\infty$ is a non full subcategory of $\ell\Op_\infty$. Given two symmetric monoidal $\infty$-categories $\cC^\otimes, \cV^\otimes$, morphisms $\cC^\otimes \to \cV^\otimes$ in $\ell\Op_\infty$ corresponds to \emph{lax} symmetric monoidal functors, while those in $\mathsf{sm}\Cat_\infty$ correspond to \emph{strong} monoidal functors.
\end{rmk}

\noindent  For symmetric monoidal $\infty$-categories $\cC^\otimes, \cV^\otimes$, let us write $\mathsf{Fun}^{\text{lax}}(\cC^\otimes,\cV^\otimes)$, resp. $\mathsf{Fun}^{\text{str}}(\cC^\otimes, \cV^\otimes)$, for the full subcategory of $\mathsf{Fun}_{{\Cat_\infty}_{/\mathsf{Fin}_*}}(\cC^\otimes, \cV^\otimes)$ spanned by lax, resp. strong, monoidal functor of symmetric monoidal $\infty$-categories.

\begin{rmk}
Any $\infty$-category $\cC$ that admits finite products has a canonical symmetric monoidal structure whose tensor product is the cartesian product (\cite[Proposition 2.4.1.5.]{Lu:HA}). In particular, this applies to the $\infty$-category of spaces $\cS$; we denote the corresponding symmetric monoidal $\infty$-category by $\cS^\times$.
\end{rmk}

\subsection{Dendroidal formalism: key concepts}\label{prelimdendr} We begin by recalling the definition of the category $\omg$ introduced in \cite{MW:DS} and discussed in detail in \cite[\S 3.2]{HeMo:SDHT}.
The objects of $\omg$ are non-planar trees $T$ with finite vertex set $V(T)$ and edge set $E(T)$, together with a specified edge, called the \emph{root} of $T$, which is attached to a single vertex. We say that $e\in E(T)$ is an \emph{inner edge} if it is both the input and output edge of two (necessarily distinct) vertices, while it is a \emph{leaf} if it is not the output edge of some vertex.

\noindent The category $\omg$ includes the object $\eta$ consisting of just one edge that is at the same time the root and the unique leaf. It also contains, for any $n\geq 0$, the $n$-corolla $C_n$, the essentially unique tree with a single vertex whose unique output edge is the root and who has precisely $n$ input edges attached.

\noindent Any tree $T$ yields a colored operad in sets $\Omega(T)$: its colors are the edges of $T$, and, for edges $e_1,\dots,e_n, e$, one has $ T(e_1,\dots,e_n;e)=\{\ast\}$ if there exists (and if it does it is unique) a subtree of $T$ with leaves $\{e_1,\dots,e_n\}$ and root $e$ and is empty otherwise; the operadic composition corresponds to grafting of subtrees.

\noindent By definition, a morphism of trees $S \to T$ in the dendroidal category $\omg$ is a morphism between the associated operads $\omg(S)\to\omg(T)$, so that $\omg$ is realized as a full sub-category of the category of discrete operads, via an embedding $ \omg \hookrightarrow \Op$. 

\noindent Alternatively, morphisms in $\omg$ are generated by four classes of morphisms:
\begin{itemize}
\item isomorphisms  $T\xrightarrow{\sim} T'$. Observe that, contrarily to $\Delta$, these can be non-trivial;
\item for every inner edge $e$ of $T$, we call \emph{elementary inner face} the map $\partial_eT\to T$ which comes from contracting $e$ in $T$ and identifying its extremal vertices. If a tree map $S \to T$ is obtained by contracting more than one inner edge, that is, as the composition of elementary inner faces, we just call it \emph{inner face}.
\item for every subtree $S$ of $T$, there is the \emph{external face} consisting of the inclusion $S\hookrightarrow T$;
\item  for every edge $e$ of $T$, there is a \emph{degeneracy} $\sigma_e T \to T$ which adds a unary vertex in the middle of $e$.
\end{itemize}

\noindent There is a fully faithful functor $i\colon \Delta\hookrightarrow \omg$, realized by sending the linear order $[n]$ to the essentially unique tree with $n+1$ edges and $n$ vertices, all of valence $1$. We identify $\Delta$ with a full subcategory of $\omg$, and under this identification faces and degeneracies have the usual meaning.

\begin{deff}\label{dendrseg1}
The $\infty$-category $\mathsf{D}\Op_\infty$ of \emph{dendroidal $\infty$-operads} is the full sub $\infty$-category of the $\infty$-category $\mathsf{PSh}(\omg)$ spanned by the presheaves satisfying the Segal and completeness properties:
\begin{enumerate}
\item  For an inner edge $e$ of $T$, calling $T_e$, resp. $T^e$, the upper part, resp. the lower part, of $T$ obtained by cutting $T$ at $e$. Then $$ X(T) \longrightarrow X(T_e) \times_{X(e)}X(T^e) $$ is an equivalence.
\item Completeness: $i^* X \in  \mathsf{PSh}(\Delta)$ is a complete Segal space.
\end{enumerate}
\end{deff}

\noindent The $\infty$-category $\mathsf{D}\Op_\infty$ is equivalent to the one presented by the model category of dendroidal complete Segal spaces, or the equivalent one of dendroidal sets (see \cite{CM:DSMHO}, \cite{HeMo:SDHT}).

\begin{rmk}
If $X\simeq i_! M$ for $M \in \mathsf{PSh}(\Delta)$, then $X$ is a dendroidal $\infty$-operad if and only if it a $\infty$-category, here in the sense of complete Segal space.
\end{rmk}

\noindent We denote by $\phi$ the category of \emph{forests}, obtained from the tree category $\omg$ by formally adjoining finite coproducts. Explicitly, the category $\phi$ is the full sub-category of the category of discrete operads spanned by $\Omega(F)$, where $F=\bigsqcup_{i=1}^n T_i$ is a finite disjoint union of trees and $\Omega(F)$ is the disjoint union of the operads $\Omega(T_i)$. We can use the category $\phi$ to reformulate \Cref{dendrseg1} in the following equivalent
\begin{deff}
The $\infty$-category $\mathsf{D}\Op_\infty$ is the full sub $\infty$-category of the $\infty$-category $\mathsf{PSh}(\phi)$ spanned by the presheaves satisfying the Segal and completeness properties, as well as the extra (also Segal-type) property:
\end{deff}
\begin{enumerate}\setcounter{enumi}{2}
\item The natural map $X(F)\to \prod_{i=1}^n X(T_i)$ for a forest $F$ consisting of the trees $T_i$ is an equivalence. In particular, $X(\emptyset)$ is contractible. 
\end{enumerate}

\noindent The Yoneda embedding induces fully faithful inclusions $$ \omg\hookrightarrow  \phi\hookrightarrow \mathsf{DOp}_\infty.$$

\begin{theorem}[\cite{CM:DSSIO}]
	The inclusion $\mathsf{D}\Op_\infty\hookrightarrow \mathsf{PSh}(\phi)$ has a left adjoint. In other words, the $\infty$-category of dendroidal $\infty$-operads $\mathsf{D}\Op_\infty$ is a localization of that of presheaves on $\phi$.
\end{theorem}

\subsection{Algebras over Lurie $\infty$-operads} Given a Lurie $\infty$-operad $\cO^\otimes$ and a symmetric monoidal $\infty$-category $\cV^\otimes$, we can talk about \emph{$\cO$-algebras in $\cV$}. We follow \cite[Definition 2.1.3.1]{Lu:HA}.

\begin{deff}
Let $\cO^\otimes$ be a $\infty$-operad and $\cV^\otimes$ be a symmetric monoidal $\infty$-category. A \emph{$\cO$-algebra in $\cV$} is a morphism of $\infty$-operads $\cO^\otimes \to \cV^\otimes$.

\noindent The $\infty$-category $\alg_{\cO^\otimes}(\cV^\otimes)$ of $\cO^\otimes$-algebras in $\cV^\otimes$ is the full subcategory of $\mathsf{Fun}_{{\Cat_\infty}_{/\mathsf{Fin}_*}}(\cO^\otimes, \cV^\otimes)$ spanned by morphisms of Lurie $\infty$-operads. 
\end{deff}

\noindent Observe that, if $\cO^\otimes \simeq \cC^\otimes$ is a symmetric monoidal $\infty$-category, then $\alg_{\cC^\otimes}(\cV^\otimes)$ is the $\infty$-category of lax monoidal functors, $\mathsf{Fun}^{\text{lax}}(\cC^\otimes,\cV^\otimes)$.

\noindent We will refer to $\mathsf{C}\alg(\cV)\coloneqq \alg_{\mathsf{Comm}^\otimes}(\cV)$ as \emph{commutative algebras in $\cV$.}

\begin{rmk}
	The notion of algebra over a dendroidal $\infty$-operad is not required for the purposes of this article, so we omit it here. For further details, the reader is referred to \cite{HeMo:SDHT}.
\end{rmk}
%
%
%

\section{The $\infty$-category of operadic left fibrations}\label{section3}

\noindent Given an $\infty$-operad $X$, whether $X$ is a dendroidal or a Lurie $\infty$-operad there are natural candidates for the $\infty$-category of operadic left fibrations over $X$. In this section, we recall the already existing notion of dendroidal left fibration for dendroidal $\infty$-operad, we define \emph{operadic left fibrations} for Lurie $\infty$-operads, and prove that they determine equivalent $\infty$-categories.

\noindent Before starting, let us recall the notion of being local with respect to a set of morphism.

\begin{deff}
	Let $\cC$ be an $\infty$-category and $x$ an object of $\cC$. Given a set $S$ of arrows of $\cC$, one says that $x$ is \emph{$S$-local} if, for any arrow $f\colon a \to b$ in $S$, the morphism $$ \mathsf{Map}_{\cC}(b,x)\xlongrightarrow{f*}  \mathsf{Map}_{\cC}(a,x)$$ is an equivalence of spaces. 
	
	\noindent Given an object $(y,f)$ in the slice $\cC_{/x}$, we say that $(y,f)$ is \emph{$S$-local} if it is $S_{/x}$-local, where $S_{/x}$ consists of those arrows in $\cC_{/x}$ of the form
	
	\begin{center}
		\begin{tikzcd}
			a \arrow[rr, "s"] \arrow[rd, "\alpha"']& & b	\arrow[ld, "\beta"] \\
			& x ,&
		\end{tikzcd}
	\end{center}
	\noindent where $s$ ranges in $S$ and $\beta$ in the arrows of $\cC$.
\end{deff}

\subsection{The $\infty$-category of dendroidal left fibrations}\label{SS:1:2}
We consider dendroidal left fibrations as first defined in \cite{BdBM:DSGSSBPQT}, where they are called \emph{covariant fibrations}; since we are working with the formalism of $\infty$-categories rather than with that of Quillen model categories, we will formulate the definitions in this language. 

\noindent Recall that, as explained in \Cref{prelimdendr}, we can realize dendroidal $\infty$-operads as a localization of the presheaf $\infty$-category $\mathsf{PSh}(\phi)$, where $\phi$ is the category of forests, obtained by freely adjoining finite coproducts to $\omg$. 
\begin{deff} For a tree $T\in \omg$, let $\ell(T)$ be the disjoint union of its leaves, and denote by $\ell(T)\hookrightarrow T$ the inclusion of these into $T$. For a forest $F$ in $\phi$, $F=T_1\sqcup \dots \sqcup T_n$, we write $\ell(F)$ for the disjoint union of the leaves of each tree in $F$, that is, $\ell(F)=\ell(T_1)\sqcup \dots \sqcup \ell(T_n)$, and $\ell(F)\hookrightarrow F$ for the inclusions of these into $F$.
\end{deff}

\noindent Let $\cL$ denote the set of morphisms in $\mathsf{PSh}(\phi)$ given by the inclusions of leaves of a forest,  $$ \cL\coloneqq \{\ell(F)\hookrightarrow F\}_{F\in \phi}$$
\begin{deff}
A morphism $f\colon Y \to X $ is a \emph{dendroidal left fibration} if $f$ is $\cL$-local in ${\mathsf{D}\Op_\infty}_{/X}$.

\noindent We write $\mathsf{D}\Left_X$ for the full sub $\infty$-category of ${\mathsf{D}\Op_\infty}_{/X}$ spanned by dendroidal left fibrations.
\end{deff}

\begin{rmk}\label{last}
The $\infty$-category $\mathsf{D}\Left_X$ is a localization of the over-category ${\mathsf{D}\Op_\infty}_{/X}$, as proven in \cite[Theorem 13.6]{HeMo:SDHT}. Moreover, it is a consequence of \cite[Lemma 13.5]{HeMo:SDHT} that, when $X=\cC$ is an $\infty$-category, there is a canonical equivalence $$\mathsf{D}\Left_\cC\simeq \Left_\cC.$$
Under the embedding $\Delta\hookrightarrow \phi$, the leaf inclusion $\ell([n])\hookrightarrow [n]$ corresponds to the map ${\{0\}\colon [0]\to [n]}$ which selects the first vertex. We recover that a map of $\infty$-categories $\cD\to \cC$ is a left fibration if and only if it is $\{\Delta^0\xhookrightarrow{\{0\}}\Delta^n\}_{n\geq 0}$-local.
\end{rmk}

\subsection{Operadic left fibrations}\label{SS:1:1} We now define the notion of left fibration of $\infty$-operads in Lurie formalism.

\noindent Denote by $U\colon \ell\Op_\infty \to \Cat_\infty$ the forgetful functor which, given an $\infty$-operad $\cO^\otimes \to \mathsf{Fin}_*$, simply returns the $\infty$-category $\cO^\otimes$, forgetting the morphism into $\mathsf{Fin}_*$.

\begin{deff}\label{opdleft}
	An \emph{operadic left fibration} is a morphism of Lurie $\infty$-operads $\cO^\otimes 	\to \cD^\otimes $ such that $Uf$ is a left fibration between $\infty$-categories.
\end{deff}

\begin{rmk}\label{crefs} The notion of operadic left fibration is equivalent to that of \emph{$\cO$-monoidal left fibration} defined in \cite[Definition 1.11]{R:MGCIC}. It differs from \cite[Definition 1.8]{KK:IOFEC}, in that instead of the condition of being a left (rather right) fibration being imposed on $\cO^\otimes \to \cD^\otimes$, it is imposed on the induced map of $\infty$-categories $\env(\cO)\to \env(\cD)$.
\end{rmk}

\noindent Given two $\infty$-operads $\cO^\otimes$, $\cD^\otimes$ and a left fibration $p\colon \cO^\otimes \to \cD^\otimes$ in ${\Cat_\infty}_{/\mathsf{Fin}_*}$, we can ask when $p$ is a morphism in $\ell\Op_\infty$. The following proposition shows that the condition to check is a Segal condition.

\begin{prop}[{\cite[Proposition 2.1.2.12]{Lu:HA}}]\label{amsterdam}
	Let $\cO^\otimes$, $\cD^\otimes$ be $\infty$-operads, and consider a morphism $p\colon \cD^\otimes \to \cO^\otimes$ in ${\Cat_\infty}_{/\mathsf{Fin}_*}$. Suppose that $p$ is a left fibration. The following properties are equivalent:
	\begin{itemize}
		\item $p$ is an operadic fibration;
		\item for any $x\in \cO^\otimes_{\underline{n}_+}$, $x\simeq x_1\oplus\dots\oplus x_n$, the inert maps $x\to x_i$ induce an equivalence of $\infty$-categories $$ \cD^\otimes_x \longrightarrow \prod_{i=1}^n \cD_{x_i}.$$
	\end{itemize}
\end{prop}

\noindent Given an $\infty$-operad $\cO^\otimes$, we denote by $\Leftlax_{\cO^\otimes}$ the full subcategory of ${\ell\Op_\infty}_{/\cO^\otimes}$ spanned by operadic left fibrations. 

\begin{rmk} After \Cref{last}, the $\infty$-category $\Leftlax_{\cO^\otimes}$ can be described as the full sub $\infty$-category of ${\ell\Op_\infty}_{/\cO^\otimes}$ spanned by the objects $(\cD^\otimes, f)$ which are $\{\Delta^0 \xrightarrow{\{0\}} \Delta^n\}_{n\geq 0}$-local after having applied the forgetful functor $\ell\Op_\infty \to {\Cat_\infty}_{/\cO^\otimes}$.
\end{rmk}

\subsection{Equivalence between dendroidal and operadic left fibrations}\label{SS:1:3}\label{unsure}

Let us start by recalling Hinich-Moerdijk comparison functors. To this end, let us recall that
\begin{itemize}
	\item the $\infty$-category $\mathsf{D}\Op_\infty$ is a localization of the presheaf $\infty$-category $\mathsf{PSh}(\phi)$, which in particular means it is equivalent to a full sub-$\infty$-category of $\mathsf{PSh}(\phi)$;
	\item the $\infty$-category $\ell\Op_\infty$ is a non-full sub $\infty$-category of ${\Cat_\infty}_{/\mathsf{Fin}_*}$, and this latter is a localization of ${\mathsf{PSh}(\Delta)}_{\mathsf{Fin}_*}\simeq \mathsf{PSh}(\mathbb{F})$, where $\mathbb{F}\coloneqq \Delta_{/\mathsf{Fin}_*}$ is the category of elements of $\mathsf{Fin}_*$, whose objects are the simplices of $\mathsf{Fin}_*$.
\end{itemize}

\noindent In \cite{HM:OELIODIO}, the authors construct a functor $$ w\colon \mathbb{F}\longrightarrow \phi,$$
\noindent based on the fact that every $n$-simplex of $\mathsf{Fin}_*$ functorially determines a $n$-leveled forest, hence a forest by forgetting the levels. Let us illustrate by some examples to build intuition about this.

\begin{itemize}
\item  A $0$-simplex of $\mathsf{Fin}_*$, that is, an element of $\mathsf{Fin}_*$, is given by the pointed set $\underline{s}_+$ for some $s\geq 0$, and we can identify it with $s$ copies of the trivial tree $\eta$.
\item  A $1$-simplex $\alpha\colon \Delta^1\to \mathsf{Fin}_*$ corresponds to a disjoint union of corollas, whose input edges are given by the set of edges given by $\alpha(0)$ and the output edges are given by $\alpha(1)$. More precisely, let $\alpha\colon \underline{m}_+\to  \underline{n}_+$ be the morphism in $\mathsf{Fin}_*$ determined by the $1$-simplex. The morphism $\alpha$ admits a essentially unique factorization as $\alpha=\alpha'\circ \alpha''$, where $\alpha'\colon \underline{k}_+\to \underline{n}_+$ is active and $\alpha''\colon \underline{m}_+\to \underline{k}_+$ is inert, and without loss of generality, we can assume that $\alpha''(i)= i$ if $i\leq k$, and $\alpha''(i)=\ast$ when $i\geq k+1$, and that $\alpha'(i)=\alpha(i)$ for all $i\leq k$. This means that the shape of the forest determined by $\alpha$ consists in $k$ corollas and $m-k$ copies of $\eta$ (i.e. the edge without any vertex).
	\[
	\begin{tikzcd}
		&                            & {} \arrow[rd, no head] & {}                                                                & {} & {} \arrow[d, no head]      & {} \arrow[d, no head] & {} \arrow[d, no head] &    \\
		{} \arrow[rrrrrrrr, no head, dashed] & \bullet \arrow[d, no head] &                        & \bullet \arrow[u, no head] \arrow[ru, no head] \arrow[d, no head] &    & \bullet \arrow[d, no head] & {}                    & {}                    & {} \\
		& {}                         &                        & {}                                                                &    & {}                         &                       &                       &   
	\end{tikzcd}
	\]
	
	\vspace*{-.3cm}
	\captionof{figure}{The forest of corollas determined by the $1$-simplex $\alpha\colon \underline{6}_+\to \underline{3}_+$ with $\alpha(1)=\alpha(2)=\alpha(3)=2$, $\alpha(4)=3$, $\alpha(5)=\alpha(6)=\ast $}
	
\item More generally, a $n$-simplex $\Delta^n \to \mathsf{Fin}_*$ consists in \emph{a $n$-leveled forest}, where edges are decorated by objects of $\cP$ and vertices by operations in $\cP^\otimes$, and where being $n$-leveled means that there is at least a tree in the forest which has a maximal branch of length $n+1$ (equivalently, a maximal chain of vertices of length $n$).
	
	\[
	\begin{tikzcd}
		{} \arrow[rd, no head]             &                                                 & {}                                             &    & {} \arrow[rd, no head] &                                                & {} \\
		{} \arrow[rrrrrr, no head, dashed] & \bullet \arrow[ru, no head] \arrow[rd, no head] &                                                & {} &                        & \bullet \arrow[ru, no head] \arrow[d, no head] & {} \\
		{} \arrow[rrrrrr, no head, dashed] &                                                 & \bullet \arrow[ru, no head] \arrow[d, no head] &    &                        & {}                                             & {} \\
		&                                                 & {}                                             &    &                        &                                                &   
	\end{tikzcd}
	\]

	\vspace*{-.3cm}
	\captionof{figure}{The forest determined by the $2$-simplex $\underline{4}_+\stackrel{\alpha_1}{\to}\underline{3}_+\stackrel{\alpha_2}{\to}\underline{1}_+$
		with $\alpha_1(1)=1=\alpha_1(2)$, $\alpha_1(3)=3=\alpha_1(4)$,
		$\alpha_2(1)=1=\alpha_2(2)$, $\alpha_2(3)=*$.
	}
	
\end{itemize}

\noindent By left Kan extension, the functor $w$ induces an adjunction $$\adjunction{w_!}{\mathsf{PSh}(\mathbb{F})}{\mathsf{PSh}(\phi)}{w^*}.$$
\noindent 

\begin{theorem}{\cite[Theorem 3.1.4]{HM:OELIODIO}}\label{victor}
	\begin{enumerate}
		\item The restriction of $w^*$  to $ \mathsf{D}\Op_\infty\subseteq\mathsf{PSh}(\phi)$ defines a functor $$\lambda\colon \mathsf{D}\Op_\infty\to \ell\Op_\infty.$$
		\item  Let  $i\colon \phi\hookrightarrow \ell\Op_\infty$ be the inclusion given by identifying $\phi$ with a full subcategory of discrete operads, and let $\delta'\colon \ell\Op_\infty \to \mathsf{PSh}(\phi)$ be the functor corresponding, under adjunction, to the functor $\mathsf{Map}_{\ell\Op_\infty}(-,i(-))\colon \ell\Op_\infty \times \phi^{\text{op}}\to \cS$. Then $\delta'$ defines a functor $$\delta\colon\ell\Op_\infty \to  \mathsf{DOp}_\infty.$$ 
		\item The above functors are part of an adjunction $$\adjunction{\delta}{\ell\Op_\infty}{\mathsf{D}\Op_\infty}{\lambda},$$ which is an equivalence of $\infty$-categories.
	\end{enumerate}

\end{theorem}


\noindent We now use this equivalence to prove the following fundamental 
 
\begin{theorem}\label{luppolo}
Let $X$ be a dendroidal $\infty$-operad, let $\cO^\otimes$ be a Lurie $\infty$-operad, with $\lambda X \simeq \cO^\otimes$.
\noindent The Hinich-Moerdijk comparison functor induces an equivalence of $\infty$-categories $$ \adjunction{\delta}{\Leftlax_{\cO^\otimes}}{\mathsf{DLeft}_X}{\lambda}.$$
\end{theorem}

\begin{proof} 
\Cref{victor} descends to an equivalence between over-$\infty$-categories $$\adjunction{\delta}{{\ell\mathsf{Op}_\infty}_{/\cO^\otimes}}{{\mathsf{DOp}_\infty}_{/X}}{\lambda}.$$
Recall that we denoted by $\cL$ the set of morphisms in $\mathsf{PSh}(\phi)$ given by $$\cL\coloneqq  \{\ell(F)\hookrightarrow F\}_{F\in \phi}, $$ and write $\cI$ for the set of morphisms in $\mathsf{PSh}(\Delta)$ given by $$ \cI\coloneqq \{\Delta^0\xhookrightarrow{\{0\}}\Delta^n\}_{n\geq 0}.$$

\noindent As dendroidal left fibrations over $X$ are $\cL$-local objects in ${\mathsf{DOp}_\infty}_{/X}$ and operadic left fibrations over $\cO^\otimes$ are objects in ${\ell\Op_\infty}_{/\cO^\otimes}$ which become $\cI$-local after post-composing with the forgetful functor ${\ell\Op_\infty}_{/\cO^\otimes}\to {\mathsf{Cat}_\infty}_{/\cO^\otimes}$, we only need to prove that, for any object $(Y,f)\in{ \mathsf{D}\Op_\infty}_{ /X}  $ with $(\cD^\otimes,\alpha^\otimes)=\lambda(Y,f)$, one has that $$ \text{$(Y,f)$ is $\cL$-local if and only if $(\cD^\otimes,\alpha^\otimes)$ is $\cI$-local in ${\Cat_\infty}_{/\cO^\otimes}$}.$$



\noindent As illustrated in \Cref{adjslice}, the adjunction $(w_!,w^*)$ induces an adjunction $$\adjunction{\tilde{w_!}}{\mathsf{PSh}(\mathbb{F})_{/\cO^\otimes}}{\mathsf{PSh}(\phi)_{/X}}{{w^*}},$$ where the left adjoint $\tilde{w_!}$ consists in applying $w_!$ and postcomposing with the counit, while the right adjoint consists in applying $w^*$.

\noindent Let $n\geq 0$ and fix $(n,p\colon \Delta^n \to \cO^\otimes)$ in ${\Cat_\infty}_{/\cO^\otimes}$. Because ${\Cat_\infty}_{/\cO^\otimes}$ is a full subcategory of $\mathsf{PSh}(\mathbbm{F})_{/\cO^\otimes}$, one has $$ \mathsf{Map}_{{\Cat_\infty}_{/\cO^\otimes}}((\Delta^n, p), (\cD^\otimes, \alpha^\otimes)) \simeq \mathsf{Map}_{{\mathsf{PSh}(\mathbb{F})}_{/\cO^\otimes}}((\Delta^n, p), (\cD^\otimes,\alpha)).$$

\noindent As $(\cD^\otimes,\alpha^\otimes)\simeq \lambda (Y,f) \simeq w^* (Y,f)$, by adjunction we obtain $$\mathsf{Map}_{\mathsf{PSh}(\mathbb{F})_{/\cO^\otimes}}((\Delta^n, p), (\cD^\otimes,\alpha))\simeq \mathsf{Map}_{\mathsf{PSh}(\phi)_{/X}}(\tilde{w_!}(\Delta^n,p), (Y,f)). $$


\noindent The object $\tilde{w_!}(\Delta^n,p)$ is of the form $$(w(\overline{p}), q\colon w(\overline{p})\to X),$$ where $\overline{p}$ is the $n$-simplex of $\mathsf{Fin}_*$ obtained as the composition $ \Delta^n \xrightarrow{p}\cO^\otimes \to \mathsf{Fin}_* $. In particular, $w(\overline{p})$ belongs to $\phi$, and hence to $\mathsf{DOp}_\infty$; as $X$ is a dendroidal $\infty$-operad and $\mathsf{D}\Op_\infty$ is full in $\mathsf{PSh}(\phi)$, one has that $\tilde{w_!}(\Delta^n,p)$ belongs to ${\mathsf{D}\Op_\infty}_{/X}$. In particular, one has the equivalence $$ \mathsf{Map}_{\mathsf{PSh}(\phi)_{/X}}((w(\overline{p}), q), (Y,f)) \simeq \mathsf{Map}_{{\mathsf{D}\Op_\infty }_{/X} }((w(\overline{p}), q), (Y,f)).$$

\noindent Consider now the inclusion $i\colon \Delta^0\xrightarrow{\{0\}} \Delta^n$ in $\cI$. Reasoning in the same way, we obtain the commutative diagram
\[
\begin{tikzcd}
\mathsf{Map}_{{\Cat_\infty}_{/\cO^\otimes}}((\Delta^n,p), (\cD^\otimes, \alpha^\otimes)) \arrow[r]\arrow[d, "\vsimeq"']&	\mathsf{Map}_{{\Cat_\infty}_{/\cO^\otimes}}((\Delta^0,pi, (\cD^\otimes, \alpha^\otimes))\arrow[d, "\vsimeq"] \\
\mathsf{Map}_{{\mathsf{D}\Op_\infty}_{/X}}((w(\overline{p}),q), (Y,f))\arrow[r] &\mathsf{Map}_{{\mathsf{D}\Op_\infty}_{/X}}((\ell(w(\overline{p})),qi), (Y,f))
\end{tikzcd}
\]
\noindent where the vertical arrows are equivalences. This shows that, if $(Y,f)$ is $\cL$-local, then $(\cD^\otimes,\alpha^\otimes))$ is $\cI$-local. After \cite[Lemma 3.1.2]{HM:OELIODIO}, we know that any forest is a retract of a forest of the form $w(\overline{p})$, which shows that if $(\cD^\otimes,\alpha^\otimes)$ is $\cI$-local, then $(Y,f)$ is $\cL$-local.

\noindent We have hence shown that $(Y,f)\in \mathsf{DOp}_\infty$ is a dendroidal left fibration if and only if $(\cD^\otimes, \alpha^\otimes)=\lambda(Y,f)$ is a operadic left fibration, and this concludes the proof.
\end{proof}


\begin{coro}\label{fibrewise1}
Let $\cO^\otimes$ be a Lurie $\infty$-operad and consider a morphism of operadic left fibrations
\[
\begin{tikzcd}
\cT^\otimes \arrow[rr, "f"] \arrow[rd, "\alpha^\otimes"']& & \cQ^\otimes \arrow[ld, "\beta^\otimes"]\\
& \cO^\otimes &
\end{tikzcd}
\] \noindent Then $f$ is an equivalence in $\Leftlax_{\cO^\otimes}$ if and only if, for any object $c$ of $\cO$, the map between fibres $$ f_c \colon (\cT,\alpha)_c \longrightarrow (\cQ, \beta)_c$$ is an equivalence of spaces.

\end{coro}

\begin{proof}
The condition is invariant under equivalence of $\infty$-categories, and by \cite[Proposition 13.8]{HeMo:SDHT}, this characterization holds in $\mathsf{D}\Left_X$ for any dendroidal $\infty$-operad $X$. Consider now $X\simeq \delta(\cO^\otimes)$; by \Cref{luppolo}, there is an equivalence $\adjunction{\lambda}{\mathsf{D}\Left_X}{\Leftlax_{\cO^\otimes}}{\delta}$, hence we have that $f\colon (\cT^\otimes, \alpha^\otimes)\to (\cQ^\otimes, \beta^\otimes)$ is an equivalence if, for any $x\in \cO^\otimes$, there is an equivalence of fibres $f_{x}\colon (\cT,\alpha)_{x}\to (\cQ,\beta)_{x}.$
By the Segal condition for $\infty$-operads, any such $x\in \cO^\otimes$ decomposes as $x\simeq c_1\oplus\dots\oplus c_n \in \cO^\otimes_{\underline{n}_+}$ for some $n$, and $f_x$ decomposes as $f_{c_1}\oplus\dots\oplus f_{c_n}$. It follows that we can equivalently check the condition in the statement only for objects $c$ of $\cO$, as wanted.
 \end{proof}

\section{Monoidal unstraightening of strong monoidal functors}\label{section7} 
\noindent Motivated by the universal property of the symmetric monoidal envelope of an $\infty$-operad, in \Cref{SS:2:2} we describe, given any symmetric monoidal $\infty$-category $\cC^\otimes$, the $\infty$-category of left fibrations over $\cC^\otimes$ whose straightening corresponds to strong monoidal functors $\cC^\otimes \to \cS^\times$. 

\noindent Let us start by recollecting the results which allow to realize lax monoidal functors as certain left fibrations of symmetric monoidal $\infty$-categories. As the result 

\subsection{Monoidal straightening-unstraightening for $\infty$-categories}\label{SS:2:1}
In \cite{C:HCHA}, Cisinski realizes the unstraightening functor for an $\infty$-category as the base change (the pullback) along the \emph{universal left fibration}, which can be identified with the forgetful functor $q_u\colon \cS_{\bullet/}\to \cS$  from the $\infty$-category of pointed spaces to  the $\infty$-category of spaces. In particular, given an $\infty$-category $\cC$, the unstraightening functor $$\unst^\cC\colon \mathsf{Fun}(\cC,\cS)\to \Left_\cC$$ sends a functor $F\colon \cC \to \cS$ to the left fibration $$ \unst^\cC(F)\simeq (S_{/\bullet}\times_{\cS} \cC, S_{/\bullet}\times_{\cS} \cC\to \cC \xrightarrow{F} \cS).$$

\noindent We will make essential use of the following important
\begin{rmk}\label{fibremanch}
For any functor $F\colon \cC\to \cS$ and any object $c$ of $\cC$, there is an equivalence of spaces $F(c)\simeq \unst(F)_c$.
\end{rmk} 

\noindent When $\cC$ is the underlying $\infty$-category of a symmetric monoidal $\infty$-category, both $\infty$-categories in the straightening-unstraightening equivalence have a symmetric monoidal structure, which on the functor category is the Day convolution. In \cite[Corollary C]{R:MGCIC}, it is proven that the un/straightening equivalence can be enhanced to a monoidal equivalence $$ \adjunction{\st^{\cC,\otimes}}{\left(\Left_\cC\right)^\otimes}{\mathsf{Fun}(\cC,\cS)^\otimes}{\unst^{\cC,\otimes}}.$$ of symmetric monoidal $\infty$-categories. In particular, the adjunction induces an equivalence between the $\infty$-categories of commutative algebras, and this corollary was independently proven also by Hinich in \cite{H:RAM}. Since in this work we only make use of the corollary and we do not use the full strength of the monoidal Grothendieck construction, let us attribute this result to Hinich and state it in the following

\begin{theorem}[{\cite[A.2]{H:RAM}}]\label{monstr}\label{lemmaalg}
For any symmetric monoidal $\infty$-category $\cC^\otimes$, the straightening-unstraightening adjunction induces an equivalence of $\infty$-categories
$$ \adjunction{\mathsf{St}^{\cC,\otimes}}{\mathsf{smLeft}_{\cC^\otimes}}{\mathsf{Fun}^\mathsf{lax}(\cC^\otimes)}{\mathsf{Unst}^{\cC,\otimes}},$$ where the $\infty$-category $\mathsf{smLeft}_{\cC^\otimes}$ is the $\infty$-category of \Cref{smleft1}.
\end{theorem} 

\noindent Throughout this article, we refer to the above adjunction as the \emph{monoidal un/straightening equivalence}. 

\noindent For sake of completeness, and to understand where the $\infty$-categories of commutative algebras of \Cref{lemmaalg} come from, let us give some detail on the monoidal $\infty$-categories  $(\mathsf{Left}_\cC)^\otimes$ and $\mathsf{Fun}(\cC,\cS)^\otimes$.

\subsubsection{Lax functors and sm-left fibrations}\label{brizzi1} Consider a symmetric monoidal $\infty$-category $\cC^\otimes$, let $\cC$ denote its underlying category. As shown in \cite[Example 2.2.6.9, Remark 4.8.1.13]{Lu:HA}), the $\infty$-category $\mathsf{Fun}(\cC, \cS)$ is the underlying $\infty$-category of a symmetric monoidal $\infty$-category, which we denote by $\mathsf{Fun}(\cC,\cS)^{\otimes_{\text{Day}}}$, whose tensor product is Day convolution. Commutative algebras in $\mathsf{Fun}(\cC, \cS)^{\otimes_{\text{Day}}}$ are lax monoidal functors, and there is an equivalence of categories $$ \mathsf{CAlg}(\mathsf{Fun}(\cC,\cS)^{\otimes_{\text{Day}}})\simeq \mathsf{Fun}^{\text{lax}}(\cC^\otimes, \cS^\times).$$

\begin{notation}
Throughout this section, we denote an object of $\mathsf{Fun}^{\text{lax}}(\cC^\otimes, \cS^\times)$ by $F^\otimes\colon \cC^\otimes \to \cS^\times$, while we write $F\colon \cC \to \cS$ for the induced functor $F^\otimes_{\underline{1}_+}$ between the underlying $\infty$-categories.
\end{notation}

\noindent In particular, the forgetful functor $$\mathsf{Fun}^{\text{lax}}(\cC^\otimes, \cS^\times)  \longrightarrow(\cC, \cS) $$ is written as $$F^\otimes \colon \cC^\otimes \to \cS^\times \mapsto F\colon \cC \to \cS.$$
 
 \begin{rmk}
The monoidal structure on spaces is cartesian; the reader can find in \cite[Proposition 2.4.1.7]{Lu:HA} an alternative description of lax functors into a cartesian symmetric monoidal $\infty$-category.


\end{rmk}

\subsubsection{Monoidal structure on $\Left_\cC$}\label{brizzi}

\noindent A symmetric monoidal  $\infty$-category $\cC^\otimes$ is a commutative algebra in $\Cat_\infty^\times$, hence ${\Cat_\infty^\times}_{/\cC^\otimes}$ is naturally a symmetric monoidal $\infty$-category, with underlying $\infty$-category ${\Cat_\infty}_{/\cC}$.

\noindent Given objects $(D_1, \gamma_1\colon D_1\to \cC), \dots,  (D_n,\gamma_n\colon D_n \to \cC), (D_\infty,\gamma\colon D_\infty\to \cC)$, the space of morphisms over $\beta\colon \underline{n}_+\to \underline{1}_+$ can be described as $$   \mathsf{Map}^\beta_{({\mathsf{Cat}_\infty}_{/\cC})^\otimes}((D_1,\gamma_1),\dots,(D_n,\gamma_n); (D_\infty,\gamma))\simeq \mathsf{Map}_{{\mathsf{Cat}_\infty}_{/\cC^{\times n}}}\left(\prod_{i=1}^n (D_i,\gamma_i), (\cC^{\times n}\times_\cC D_\infty,\Gamma)\right),$$  where $\Gamma$ is the map in the cartesian diagram
\[
\begin{tikzcd}
\cC^{\times n}\times_\cC D_\infty \arrow[d, "\Gamma"'] \arrow[rr] &                                                                         & D_\infty \arrow[d, "\gamma"] \\
\cC^{\times n}                                         & \cC^\otimes_{\underline{n}_+} \arrow[l, "\simeq"'] \arrow[r, "\beta_!"] & \cC                                
\end{tikzcd}
\]

\noindent The $\infty$-category $\Left_\cC$ of left fibrations over $\cC$ has a symmetric monoidal structure as well, and the corresponding symmetric monoidal $\infty$-category is given by $\Left_\cC^\otimes$. After \cite[Theorem 4.2]{R:MGCIC}, it can be identified with a non-full sub-$\infty$-operad of $(\Cat_\infty/\cC)^\otimes$, and, given left fibrations $(D_1, \gamma_1\colon D_1\to \cC), \dots,  (D_n,\gamma_n\colon D_n \to \cC)$, the forgetful functor identifies the space of multimorphisms of $\Left_{\cC}^\otimes$ over $\beta$, i.e. 
$$ \mathsf{Map}_{(\Left_{\cC})^\otimes}((D_1,\gamma_1),\dots,(D_n,\gamma_n);(D_\infty,\gamma))_{\beta},$$ with the full subcategory of $$  \mathsf{Map}_{({\mathsf{Cat}_\infty}_{/\cC})^\otimes} ((D_1,\gamma_1),\dots,(D_n,\gamma_n); (D_\infty,\gamma))_{\beta}$$ spanned by the components corresponding to functors $\prod_{i=1}^n D_i \to D_\infty$ lying over $\beta_! \colon \cC^{\times n}\to \cC$.

 \begin{deff}\label{smleft1}
 	A \emph{sm-left fibration } is a morphism $f\colon \cC^\otimes \to \cD^\otimes$ in $\mathsf{smCat}_\infty$ such that $Uf$ is a left fibration between $\infty$-categories.
 	\end{deff}

\noindent We can look at sm-left fibrations over a fixed symmetric monoidal $\infty$-category $\cC^\otimes$. We denote by $\mathsf{sm}\Left_{\cC^\otimes}$ the full subcategory of ${\mathsf{smCat}_{\infty}}_{/\cC^\otimes}$ spanned by sm-left fibrations over $\cC^\otimes$.

\begin{rmk}
As there is a forgetful functor $\mathsf{sm}\Cat_\infty\to \ell\Op_\infty$, given a symmetric monoidal $\infty$-category $\cC^\otimes$, there is also a forgetful functor $ \mathsf{sm}\Left_{\cC^\otimes}\to \Leftlax_{\cC^\otimes}$.
\end{rmk}
 
\noindent We can characterize commutative algebras in $\Left_\cC^\otimes$ as a non-full subcategory of ${\mathsf{sm}\Cat_\infty}_{/\cC^\otimes}$: they are the sm-left fibrations over $\cC^\otimes$, as stated in the following 

\begin{theorem}[{\cite{H:RAM}}]\label{questionable}
For any symmetric monoidal $\infty$-category $\cC^\otimes$, there is an equivalence of $\infty$-categories $$ \mathsf{CAlg}((\Left_{\cC})^\otimes) \simeq \mathsf{sm}\Left_{\cC^\otimes}.$$
\noindent Under this equivalence, the forgetful functor $$\mathsf{CAlg}((\Left_{\cC})^\otimes)\simeq\mathsf{sm}\Left_{\cC^\otimes}\to \Left_{\cC}$$
sends a sm-left fibration to the left fibration between the underlying categories, namely $$(\cO^\otimes, \alpha^\otimes \colon \cO^\otimes \to \cC^\otimes) \mapsto (\cO, \alpha\colon \cO\to \cC).$$
\end{theorem}
%

\subsection{Unstraightening of strong monoidal functors}\label{SS:2:2} Motivated by the universal property of the symmetric monoidal envelope of a Lurie $\infty$-operad (to be addressed systematically in \Cref{smenv}), in the next proposition we characterize, for any symmetric monoidal $\infty$-category $\cC^\otimes$, the full subcategory of $\mathsf{sm}\Left_{\cC^\otimes }$ whose monoidal straightening corresponds to \emph{strong} monoidal functors $\cC^\otimes \to \cS^\times$.

\begin{prop} \label{tiprego}\label{malaga} 
Let $\cC^\otimes$ be a symmetric monoidal $\infty$-category, and let $(\cD^\otimes, \alpha^\otimes)$ be an object of $\mathsf{sm}\Left_{\cC^\otimes}$. The following conditions are equivalent: 
\begin{enumerate}
	\item  $(\cD^\otimes, \alpha^\otimes)\simeq \unst^{\cC,\otimes}(F^\otimes)$ for a \emph{strong} monoidal functor $F^\otimes \colon \cC^\otimes 	\to \cS^\times$;
	\item for any $x\in \cC^\otimes_{\underline{n}_+}$, $x\simeq x_1\oplus\dots\oplus x_n $, the induced map between the fibres $$ (\cD^\otimes,\alpha^\otimes)_{x_1\oplus\dots\oplus x_n}\longrightarrow (\cD^\otimes,\alpha^\otimes)_{x_1\otimes\dots\otimes x_n}$$ induced by $\beta_! \colon x\to x_1\otimes \dots \otimes x_n$ in $\cC^\otimes$ is an equivalence.
\end{enumerate}
\end{prop}

\begin{rmk}
By \Cref{amsterdam}, the condition of \Cref{tiprego} is equivalent to asking that the morphism $$  \prod_{i=1}^n (\cD,\alpha)_{x_i}\xleftarrow{\sim} (\cD^\otimes,\alpha^\otimes)_{x_1\oplus\dots\oplus x_n}\longrightarrow (\cD^\otimes,\alpha^\otimes)_{x_1\otimes \dots \otimes x_n}$$ is an equivalence.
\end{rmk}

\begin{proof} \noindent Consider an element $(\cD^\otimes, \alpha^\otimes)$ in $\mathsf{sm}\Left_{\cC^\otimes}$, and let $F^\otimes \in \mathsf{Fun}^{\text{lax}}(\cC^\otimes, \cS^\times)$ be such that $(\cD^\otimes,\alpha^\otimes) \simeq \unst^{\cC,\otimes}(F^\otimes)$. The lax monoidal functor $F^\otimes \colon \cC^\otimes \to \cS^\times$ is \emph{strong} if it preserves all coCartesian lifts. Since inert-active morphisms in $\mathsf{Fin}_*$ form a factorization system and $F^\otimes$ is lax monoidal, it is sufficient to check that $F^\otimes$ preserves coCartesian lifts of active morphisms, and it is actually enough to see that it preserves coCartesian lifts of $\beta\colon \underline{n}_+\to \underline{1}_+$.

\noindent Let $x_1,\dots,x_n$ be objects of $\cC$, let $\beta_! \colon x_1\oplus\dots\oplus x_n \to x_1\otimes \dots \otimes x_n $ be the coCartesian lift of $\beta$ over $x_1\oplus\dots\oplus x_n$, and consider $F^\otimes(\beta_!)\colon F^\otimes(x_1\oplus\dots\oplus x_n)\to F(x_1\otimes \dots \otimes x_n)$, where $F\colon \cC \to \cS$ is the functor between the underlying $\infty$-categories.
 Since $F^\otimes$ is lax, the projections induce an equivalence $F^\otimes(x_1\oplus\dots\oplus x_n) \xrightarrow{\simeq } \prod_{i=1}^n F(x_i).$ Then $F^\otimes(\beta_!)$ is equivalent to the coCartesian lift in $\cS^\times$ of $\beta$ over $F(x_1)\times \dots \times F(x_n),$ if and only if the morphism 
\begin{equation}\label{laxity}
F(x_1)\times \dots \times F(x_n)\xleftarrow{\simeq} F^\otimes_{\underline{n}_+}(x_1\oplus \dots \oplus x_n)\longrightarrow F(x_1\otimes \dots \otimes x_n) ,
\end{equation} is an equivalence.

\noindent On the other hand, by taking the fibre over $x_1\otimes \dots \otimes x_n$ of the $n$-fold algebra map of $(\cO^\otimes,\alpha^\otimes)$
we obtain the following map
\begin{equation}\label{trabi}
(\cD,\alpha)_{x_1}\times \dots \times (\cD,\alpha)_{x_n} \xleftarrow{\simeq} (\cD^\otimes,\alpha^\otimes)_{x_1\oplus \dots \oplus x_n}  \xlongrightarrow{\beta_!} (\cD,\alpha)_{x_1\otimes \dots \otimes x_n}.
\end{equation}

\noindent Since $(\cD^\otimes,\alpha^\otimes)\simeq \unst^{\cC,\otimes}(F^\otimes)$ and as observed in \Cref{fibremanch} for any object $x$ of $\cC$, there is an equivalence $(\cD,\alpha)_{x}\simeq F(x),$ one can see that the morphisms in \Cref{trabi} and \Cref{laxity} are equivalent. This means that $F^\otimes$ is strong monoidal if and only if $(\cD^\otimes, \alpha^\otimes)$ satisfies the condition in $(2)$, as wanted.

\end{proof}

\begin{deff}
	For a symmetric monoidal $\infty$-category $\cC^\otimes$, we write $\mathsf{sm}\Leftbeta_{\cC^\otimes}$ for the full $\infty$-subcategory of $\mathsf{sm}\Left_{\cC^\otimes}$ spanned by those sm-left fibrations satisfying the condition of \Cref{malaga}. 
\end{deff}

\noindent We now characterize equivalences in the category $\mathsf{sm}\Leftbeta_{\cC^\otimes}$ as those maps which are fibrewise equivalences.

\begin{prop}\label{fibrewise2}
Let $\cC^\otimes$ be a symmetric monoidal $\infty$-category, and consider a morphism of sm-left fibrations
\[
\begin{tikzcd}
\cD^\otimes \arrow[rr, "f^\otimes"] \arrow[rd, "\gamma^\otimes"']& & \cE^\otimes \arrow[ld, "\vartheta^\otimes"]\\
& \cC^\otimes & 
\end{tikzcd}
\] \noindent Then $f^\otimes$ is an equivalence in the $\infty$-category $\mathsf{sm}\Leftbeta_{\cC^\otimes}$ if and only if, for any object $x$ of $\cC$, the map between fibres $$ f_x \colon (\cD,\gamma)_x \longrightarrow (\cE, \vartheta)_x$$ is an equivalence of spaces.
\end{prop}

\begin{proof}
The $\infty$-category $\mathsf{sm}\Leftbeta_{\cC^\otimes}$ is a full subcategory of $\mathsf{sm}\Left_{\cC^\otimes}$, hence it suffices to prove that $f$ is an equivalence in this latter. As $ \mathsf{sm}\Left_{\cC^\otimes}\simeq \mathsf{CAlg}(\Left_{\cC}^\otimes)$, the equivalences in $ \mathsf{sm}\Left_{\cC^\otimes}$ are detected by the forgetful functor $$\mathsf{CAlg}(\Left_{\cC}^\otimes)\to \Left_{\cC}, \quad (\cF^\otimes, \eta^\otimes)\mapsto (\cF,\eta),$$ and this implies that $f^\otimes$ is an equivalence if and only if $f\colon (\cD, \gamma)\to (\cE, \vartheta)$ is an equivalence in $\Left_{\cC}$, which is true if and only if, for any $x\in \cC$, the morphism between fibres $$f_x\colon (\cD,\gamma)_x\to (\cE, \vartheta)_x$$ is an equivalence, and this concludes the argument.
\end{proof}

\section{Symmetric monoidal envelope of operadic left fibrations}\label{smenv}

\noindent We now work with the symmetric monoidal envelope of an $\infty$-operad and the induced adjunction between slice categories; we recall the definitions in \Cref{round1}. In \Cref{ranianna}, we prove that strong sm-left fibrations are equivalent to operadic left fibrations in two steps: first, showing the symmetric monoidal envelope defines a fully faithful left adjoint in \Cref{pollo}, then proving its essential image coincides with strong sm-left fibrations in \Cref{kriku}, relying on \Cref{luppolo}.

\subsection{Symmetric monoidal envelope}\label{round1}

\noindent Let $\mathsf{Act}( \mathsf{Fin_*})$ be the nerve of the full subcategory of $\mathsf{Fun}([1],\mathsf{Fin}_*)$ spanned by active morphisms. We recall the definition of Lurie monoidal envelope of an $\infty$-operad (\cite[\S 2.2.4]{Lu:HA}).

\begin{deff}
Let $\cO^\otimes \to \mathsf{Fin_*}$ be an $\infty$-operad. We write $\env(\cO)^\otimes$ for the fibre product $$ \cO^\otimes \times_{\mathsf{Fun}(\{0\}, \mathsf{Fin_*})} \mathsf{Act}( \mathsf{Fin_*}).$$ The target inclusion $\Delta^0 \xrightarrow{\{1\}} \Delta^1$ induces a map $\env(\cO)^\otimes \to \mathsf{Fin}_*$, so $\env(-)^\otimes$ defines a functor $$ \env(-)^\otimes \colon \mathsf{Op}_\infty \longrightarrow {\mathsf{Cat}_\infty}_{/\mathsf{Fin}_*}.$$
\noindent Denote by $\env(\cO)$ the fibre of $\env(\cO)^\otimes \to \mathsf{Fin}_*$ over $\underline{1}_+$.
\end{deff}

\noindent The symmetric monoidal envelope of a $\infty$-operad can be characterized via the following universal property.

\begin{prop}[{ \cite[Proposition 2.2.4.9]{Lu:HA}}]\label{freepalestine}
	The envelope functor $\env(-)^\otimes$ takes values in the $\infty$-category of symmetric monoidal $\infty$-categories with strong monoidal functors and is left adjoint to the forgetful functors. In symbols, $$\adjunction{\env(-)^\otimes}{\ell\Op_\infty}{\mathsf{smCat}_\infty}{U}.$$
\end{prop}

\noindent Equivalently, we can reformulate the above result by saying that, for any $\infty$-operad $\cO^\otimes$ and any symmetric monoidal $\infty$-category $\cV^\otimes$, there is a natural equivalence of $\infty$-categories $$ \alg_{\cO}(\cV)\simeq \mathsf{Fun}^{\text{str}}(\env(\cO)^\otimes, \cV^\otimes).$$ 

\begin{rmk}\label{descriptionenv}
	There is an equivalence of $\infty$-categories $$\env(\cO)\simeq \cO^{\otimes,act},$$ where $\cO^{\otimes,act}$ is the wide subcategory of $\cO^\otimes$ spanned by active morphisms; in particular,
	an object of $\env(\cO)$ corresponds to an element $\underline{c}\in \cO^\otimes_{\underline{m}_+}$, that is, a list of objects of $\cO$.

	\noindent More generally, given $n\geq 0$, an object $x$ in the fibre $\env(\cO)^\otimes_{\underline{n}_+}$ writes as  $x\simeq (\underline{c}, \alpha)$, with $\underline{c}$ in $ \cO^\otimes_{\underline{m}_+}$ for some $m\geq 0$ and $\alpha$ an active morphism $\alpha\colon \underline{m}_+\to \underline{n}_+$. The Segal condition can be expressed in the following way: given elements $\underline{d}^1,\dots,\underline{d}^n$ in $\env(\cO)$, $\underline{d}^i \in \cO^\otimes_{\underline{m_i}_+}$ for some $\underline{m_i}'s$, the object $ \underline{d}^1\oplus\dots\oplus \underline{d}^n$ in $ \env(\cO)^\otimes_{\underline{n}_+}$ corresponds to $$  \underline{d}^1\oplus\dots\oplus \underline{d}^n \simeq ( \underline{d}, \alpha\colon \underline{m}_+\to \underline{n}_+),$$ where $\underline{d}\simeq \underline{d}^1\oplus \dots \oplus \underline{d}^n $ as an element of $ \cO^\otimes_{\underline{m}_+}\simeq \underset{i=1}{\overset{n}{\prod}}\cO^\otimes_{\underline{m^i}_+}$ and $\alpha$ appropriately partitions $\underline{m}_+$. If we regard an object of $\env(\cO)$ as a list of objects of $\cO$, the tensor product in the envelope consists in the concatenation of lists: given $\underline{d}^i$'s as above, we represent the tensor product of the $\underline{d}^i$'s by the morphism 
	$$\beta_!\colon   \underline{d}^1\oplus\dots\oplus\underline{d}^n \to \underline{d}^1\otimes \dots\otimes\underline{d}^n $$
	$$ (\underline{d}, \alpha\colon \underline{m}_+\to \underline{n}_+)\xlongrightarrow{(\id,\beta)} (\underline{d}, \beta\colon \underline{m}_+ \to \underline{1}_+).$$ 

\end{rmk}

\noindent Neither the symmetric monoidal envelope nor the forgetful functor of \Cref{freepalestine} are fully faithful. However, the functor $\env(-)^\otimes $ becomes fully faithful when seen as a functor between over-$\infty$-categories. The following proposition is a consequence of  \cite[Proposition 2.4.3]{HK:IOSMIC}.

\begin{prop}\label{nonna} Let $\cO^\otimes$ be a $\infty$-operad, and consider the adjunction
	$$ \adjunction{\env(-)^\otimes}{{\ell\Op_\infty}_{/\cO^\otimes}}{{\mathsf{sm}\Cat_\infty}_{/\env(\cO)^\otimes}}{G'},$$ where $G'$ is the base-change along the unit after having applied $U$. The symmetric monoidal envelope, seen as a functor $\env(-)^\otimes \colon {\ell\Op_\infty}_{/\cO^\otimes}\to {\mathsf{sm}\Cat_\infty}_{/\env(\cO)^\otimes}$, is fully faithful.
\end{prop}

\subsection{Symmetric monoidal envelope of operadic left fibrations}\label{ranianna}

\begin{lemma}\label{pollo}

For any $\infty$-operad $\cO^\otimes$, the adjunction in \Cref{nonna} restricts to an adjunction
$$\adjunction{\env(-)^\otimes}{\Leftlax_{\cO^\otimes} }{\mathsf{sm}\Left_{\env(\cO)^\otimes}}{G}.$$
\end{lemma}

\begin{proof} 
The categories $ \Leftlax_{\cO^\otimes}$, resp. $ \mathsf{sm}\Left_{\env(\cO)^\otimes}$, are full subcategories of ${\ell\mathsf{Op}_\infty}_{/\cO^\otimes}$, resp. ${\mathsf{smCat}_\infty}_{/\env(\cO)^\otimes}$, so we only need to see that the restrictions of $\env(-)^\otimes$ and $G'$ are well defined. 

\noindent To be completely explicit, recall that, given a sm-left fibration $(\cT^\otimes,\alpha^\otimes)$, its image via $G'$ is the left vertical arrow in the following pullback diagram:
\begin{center}
	\begin{tikzcd}
		\cO^\otimes \times_{\env(\cO)^\otimes}\cT^\otimes \arrow[d, "{G'(\cT^\otimes,\alpha^\otimes)}"'] \arrow[r] & \cT^\otimes \arrow[d, "\alpha^\otimes"] \\
		\cO^\otimes  \arrow[r, "\iota_\cO"']                                                               & \env(\cO)^\otimes                      
\end{tikzcd}\end{center}

\noindent Let $(\cP^\otimes, f)$ be an operadic left fibration. There is a commutative diagram
\begin{center}
\begin{tikzcd}
\env(\cP)^\otimes \arrow[d, "\env(f)^\otimes"'] \arrow[r] & \cP^\otimes \arrow[d, "f"] \\
\env(\cO)^\otimes \arrow[d] \arrow[r] & \cO^\otimes \arrow[d]\\
\mathsf{Act}(\mathsf{Fin_*}) \arrow[r] & \mathsf{Fin_*}  \ ,
\end{tikzcd}
\end{center}
\noindent where the outer square and the lower square are cartesian. By the pasting lemma, the upper square is cartesian as well, and since left fibrations are stable under pullback, $\env(f)^\otimes$ is also a left fibration. This means that $\env(\cP^\otimes, f^\otimes)\simeq (\env(\cP)^\otimes, \env(f)^\otimes)$ belongs to $\mathsf{sm}\Left_{\env(\cO)^\otimes}$, as wanted.

\noindent Given a sm-left fibration $(\cT^\otimes, \varphi^\otimes)$ in $\mathsf{sm}\Left_{\env(\cO)^\otimes}$, the fact that left fibrations of $\infty$-categories are stable under pullback ensures that the morphism $G(\cT^\otimes,\varphi^\otimes)$ is also a left fibration in $\Cat_\infty$. As it was already an element in $\ell\Op_\infty$, we conclude that $G(\cT^\otimes,\varphi^\otimes)$ is an operadic left fibration. This concludes the proof.
\end{proof}

\noindent We refine the previous result, and showing that the symmetric monoidal envelope establish an equivalence between operadic left fibrations and \emph{strong} sm-left fibrations.

\begin{prop}\label{kriku}
For any $\infty$-operad $\cO^\otimes$, the adjunction in \Cref{pollo} restricts to an adjunction of $\infty$-categories $$\adjunction{\env(-)^\otimes}{\Leftlax_{\cO^\otimes} }{\mathsf{sm}\Leftbeta_{\env(\cO)^\otimes}}{G}.$$ Moreover, this adjunction is an equivalence.
\end{prop}

\begin{proof} Consider a operadic left fibration $(\cD^\otimes, \alpha^\otimes)$. To see that its image $(\env(\cD)^\otimes, \env(\alpha)^\otimes)$ belongs to $\mathsf{sm}\Leftbeta_{\env(\cO)^\otimes}$, we apply \Cref{malaga} and check that, for any $\underline{c} \in \env(\cO)^\otimes_{\underline{n}_+}$, $\underline{c}\simeq \underline{c}^1 \oplus\dots\oplus\underline{c}^n$ for some $\underline{c}^i \in \cO^\otimes_{\underline{m_i}_+}\subseteq \env(\cO)$, the morphism
$$ \beta_!\colon (\env(\cD)^\otimes, \env(\alpha)^\otimes)_{\underline{c}^1 \oplus\dots\oplus\underline{c}^n }\longrightarrow (\env(\cD),\env(\alpha))_{\underline{c}^1 \otimes\dots\otimes\underline{c}^n}$$ induced by the coCartesian lifts of $\beta_!\colon \underline{c}^1 \oplus\dots\oplus\underline{c}^n  \to \underline{c}^1 \otimes\dots\otimes\underline{c}^n $ is an equivalence.

\noindent Since the fibres of $\alpha^\otimes \colon \cD^\otimes\to \cO^\otimes$ lie in $\cD^{\otimes,act}$ and $\env(\cD)\simeq \cD^{\otimes,act}$, for any $z\in \cO^\otimes_{\underline{m}_+}\subseteq \env(\cO)$ there is an equivalence $$(\env(\cD),\env(\alpha))_z\simeq (\cD^{\otimes,act},\alpha^{\otimes,act})_z \simeq (\cD^\otimes,\alpha^\otimes)_z.$$

\noindent  By \Cref{amsterdam}, the inert maps $\underline{c}\to \underline{c}^i$ in $\env(\cO)^\otimes$ induce an equivalence 
$$ (\env(\cD)^\otimes, \env(\alpha)^\otimes)_{\underline{c}^1 \oplus\dots\oplus\underline{c}^n }\xrightarrow{\sim} \prod_{i=1}^n  (\env(\cD), \env(\alpha))_{\underline{c}^i}\simeq \prod_{i=1}^n(\cD^\otimes,\alpha^\otimes)_{\underline{c}^i}.$$ As a consequence, $\beta_!$ fits into the commutative diagram 

\[
\begin{tikzcd}
(\env(\cD)^\otimes,\env(\alpha)^\otimes)_{\underline{c}^1 \oplus\dots\oplus\underline{c}^n } \arrow[d, "\sim"]\arrow[rr, "\beta_!"]& & (\env(\cD),\env(\alpha))_{\underline{c}^1 \otimes\dots\otimes\underline{c}^n}\arrow[d,"\simeq"] \\
  \prod_{i=1}^n  (\env(\cD),\env(\alpha))_{\underline{c}^i}  \arrow[r,"\simeq"]& \prod_{i=1}^n (\cD^{\otimes},\alpha^\otimes)_{\underline{c}_i} & \arrow[l, "\psi"'](\cD^{\otimes},\alpha^\otimes)_{\underline{c}^1\oplus\dots \oplus \underline{c}^n}
\end{tikzcd}
\]
\noindent and $\beta_!$ is an equivalence if and only if $\psi$ is, so let us prove this last fact. For every $i\in \{1,\dots,n\}$, we can decompose $\underline{c}^i$ as an object in $\cO^\otimes$ and write $\underline{c}^i\simeq c^i_{j_1}\oplus\dots\oplus c^i_{j_i}$ for some $c^i_j \in \cO$, and by \Cref{amsterdam} the coCartesian lifts of inerts in $\cO$ induce an equivalence $$(\cD^\otimes,\alpha^\otimes)_{\underline{c}^1\oplus\dots\oplus\underline{c}^n}\xlongrightarrow{\sim} \prod_{i,j} (\cD,\alpha)_{c^i_j}.$$ 
This equivalence factors through $\psi$, written as the composition $$ (\cD^\otimes,\alpha^\otimes)_{\underline{c}^1\oplus\dots\oplus\underline{c}^n}\xlongrightarrow{\psi} \prod_{i=1}^n (\cD^\otimes,\alpha^\otimes)_{\underline{c}^i}\xlongrightarrow{\simeq} \prod_{i,j} (\cD,\alpha)_{c^i_j},$$ where the second map is an equivalence again because of \Cref{amsterdam}, so we conclude that $\psi$ is an equivalence. This means that $G(\cD^\otimes,\alpha^\otimes)$ lies in $\mathsf{sm}\Leftbeta_{\env(\cO)^\otimes}$, as wanted. 

\noindent We have shown that the adjunction in the statement is well defined, so let us now prove that it is an equivalence of $\infty$-categories. 

\noindent From \cite[Proposition 2.4.3.]{HK:IOSMIC}, the symmetric monoidal envelope is a fully faithful functor when regarded as a functor $\env(-)^\otimes \colon \ell\Op_\infty\simeq {\ell\Op_\infty}_{/\mathsf{Comm}^\otimes} \longrightarrow {\mathsf{sm}\Cat_\infty}_{/\env(\mathsf{Comm})^\otimes}$. In particular, its restriction $\env(-)^\otimes\colon \Leftlax_{\cO^\otimes} \longrightarrow \mathsf{sm}\Left_{\env(\cO)^\otimes}$ is fully faithful as well.

\noindent By adjunction, to show that $\env(-)^\otimes$ is also essentially surjective, it is enough to prove that $G$ is conservative, that is, it reflects weak equivalences.

\noindent Consider a morphism of sm-left fibrations $f\colon X \to Y$ such that $G(f)\colon GX \to GY$ is an equivalence. By \Cref{fibrewise1}, this is equivalent to asking that, for any $c\in \cO^\otimes_{\underline{n}_+}$, for any $n\geq 0$, the map between the fibres $(G(f))_c\colon GX_c \to GY_c$ is an equivalence of spaces. 
\noindent On the other hand, by \Cref{fibrewise2}, $f\colon X \to Y$ is an equivalence if and only if, for any object $d$ of $\env(\cO)$, the map of fibres $f_d\colon X_d \to Y_d$ is an equivalence of spaces. 
\noindent We conclude by observing that, for any $c\in \cO^\otimes_{\underline{n}_+}\subseteq \env(\cO)$, one has the equivalence $GX_c \simeq X_c$, and that any $d\in \env(\cO)$ is of the form $d=c\in \cO^\otimes_{\underline{n}_+}$ for some $n$.
\end{proof}

\section{The un/straightening equivalence for $\infty$-operads}\label{triku} 

\noindent We have gathered everything we need to state and prove our straightening-unstraightening theorem for $\infty$-operads. Before that, let us summarize the $\infty$-categories of left fibrations we have introduced so far:
\begin{itemize}
	\item For a Lurie $\infty$-operad $\cO^\otimes$, the $\infty$-category $\Leftlax_{\cO^\otimes}$ of operadic left fibrations over $\cO^\otimes$ is the full sub $\infty$-category of ${\ell\Op_\infty}_{/\cO^\otimes}$ spanned by the elements $(\cD^\otimes, p)$ such that $p\colon \cD^\otimes \to \cO^\otimes$ is a left fibration of $\infty$-categories.
	
	\noindent An object $(\cD^\otimes, p)$ in ${\Cat_\infty}_{/\cO^\otimes}$ belongs to $\Leftlax_{\cO^\otimes}$ if and only if $p$ satisfies the lax monoidality condition of \Cref{amsterdam}.
	
	\item For a symmetric monoidal $\infty$-category $\cC^\otimes$, the $\infty$-category $\mathsf{sm}\Left_{\cC^\otimes}$ of sm left fibrations over $\cC^\otimes$ is the full sub $\infty$-category of ${\mathsf{sm}\Cat_\infty}_{/\cC^\otimes}$ spanned by the elements $(\cD^\otimes, p)$ for which $p$ is a left fibration of $\infty$-categories. It is equivalent to the $\infty$-category of commutative monoids in $\Left_\cC^\otimes$.
	
	\item The $\infty$-category $\mathsf{sm}\Leftbeta_{\cC^\otimes}$ of strong sm left fibrations is the full sub $\infty$-category of $\mathsf{sm}\Left_{\cC^\otimes}$ spanned by the sm left fibrations $(\cD^\otimes, p)$ satisfying the condition in \Cref{fibrewise2}. 
\end{itemize}

\begin{theorem}\label{mainnnn}
	For any Lurie $\infty$-operad $\cO^\otimes$, there is an equivalence of $\infty$-categories $$ \adjunction{\st^\cO}{\Leftlax_{\cO^\otimes}}{\alg_{\cO^\otimes}(\cS)}{\unst^\cO},$$ where the left adjoint is given by the composition $$ \st^{\cO}\colon \Leftlax_{\cO^\otimes} \xrightarrow{\env(-)^\otimes}\mathsf{sm}\Leftbeta_{\env(\cO)^\otimes} \xrightarrow{\st^{\env(-),\otimes}} \mathsf{Fun}^{\text{str}}(\env(\cO)^\otimes,\cS^\times)\simeq \alg_{\cO^\otimes}(\cS^\times).$$ 
	
	\noindent In other words, if $(\cT^\otimes, \alpha^\otimes\colon \cT^\otimes \to \cO^\otimes)$ is an operadic left fibration, and $x$ is an object of $\cO$, the value at $x$ of the $\cO^\otimes$-algebra $\st^{\cO^\otimes}(\cT^\otimes, \alpha^\otimes)$ is the space described as $$ \st(\cT^\otimes, \alpha^\otimes)(x)\simeq \env(\cT)\times_{\env(\cO)} \env(\cO)/x.$$ Conversely, given a $\cO^\otimes$-algebra $F$, modeled as a strong monoidal functor $F^\otimes \colon \env(\cO)^\otimes \to \cS^\times$, the operadic left fibration $\unst^\cO(F^\otimes)$ is obtained as the pullback $$ \unst^\cO(F^\otimes)\simeq ({\cS_{\bullet/}}^\times \times_{\cS^\times} \cO^\otimes, \cS_{\bullet/}^\times \times_{\cS^\times} \cO^\otimes \longrightarrow \cO^\otimes),$$ where ${\cS_{\bullet/}}\to \cS$ is the \emph{universal left fibration}, that is, the forgetful functor from pointed spaces to spaces, strong monoidal with respect to the cartesian product.
\end{theorem}

\begin{proof}
Consider first the adjunction given by the restriction to operadic left fibrations of the symmetric monoidal envelope: this yields an adjunction  $$\adjunction{\env(-)^\otimes}{\Leftlax_{\cO^\otimes}}{\mathsf{sm}\Leftbeta_{\env(\cO)^\otimes}}{G}, $$ which we have shown in \Cref{kriku} to be an equivalence of $\infty$-categories. 
\noindent Since $\env(\cO)^\otimes$ is a symmetric monoidal category, the straightening-unstraightening equivalence is monoidal, and in particular it induces an equivalence between commutative algebras, which acts as the straightening-unstraightening functor on the underlying objects. We consider the restriction of this latter to sm-left fibrations over the envelope of $\cO^\otimes$ and strong monoidal functors out of $\env(\cO)^\otimes$; we obtain another adjunction $$\adjunction{\st^{\env(\cO),\otimes}}{\mathsf{sm}\Leftbeta_{\env(\cO)^\otimes}}{\mathsf{Fun}^{\text{str}}(\env(\cO)^\otimes,\cS^\times)\simeq \alg_{\cO^\otimes}(\cS)}{\unst^{\env(\cO),\otimes} },$$ which, thanks to \Cref{malaga}, is an equivalence as well. 

\noindent The straightening functor is the composition of the two left adjoints just mentioned, and the unstraightening functor is the composition of the right adjoints. This yields an adjunction  $$ \adjunction{\st^\cO}{\Leftlax_{\cO^\otimes}}{\mathsf{Alg}_{\cO^\otimes}(\cS^\times)}{\unst^\cO},$$ which is an equivalence, as wanted.
\end{proof}

\noindent Exploiting the strictification of the monoidal straightening equivalence for discrete categories proven in \cite[Theorem 4.4]{P:RDLF}, and more precisely the reformulation of \cite[Corollary 5.3]{P:RDLF}, we can also state the following explicit formula for the operadic straightening functor in the special case $\cO^\otimes$ is discrete.

\begin{coro}\label{coromoro}
	For a \emph{discrete} $\infty$-operad $\cO^\otimes$, the operadic straightening equivalence can be explicitely written in the following terms. Given an operadic left fibration $(\cT^\otimes, \alpha^\otimes)$ over $\cO^\otimes$ and an object $x$ of $\cO$, the value at $x$ of the $\cO^\otimes$-algebra $\st^{\cO}(\cT^\otimes, \alpha^\otimes)$ is given by $$ \st^\cO(\cT^\otimes, \alpha^\otimes)(x)\simeq \env(\cT)\times_{\env(\cO)} {\env(\cO)}_{/x}.$$ 
\end{coro}

\addtocontents{toc}{\SkipTocEntry}

\providecommand{\bysame}{\leavevmode\hbox to3em{\hrulefill}\thinspace}
\providecommand{\MR}{\relax\ifhmode\unskip\space\fi M`R }
\providecommand{\MRhref}[2]{%
  \href{http://www.ams.org/mathscinet-getitem?mr=#1}{#2}}
\providecommand{\href}[2]{#2}

\bibliographystyle{alpha}
\bibliography{Part_II_oo_v2.bib}

\begin{thebibliography}{BdBM20}

\bibitem[Bar18]{B:FOCHO}
Clark Barwick.
\newblock From operator categories to higher operads.
\newblock {\em Geometry \& Topology}, 22(4):1893--1959, 2018.
\newblock Publisher: Mathematical Sciences Publishers.

\bibitem[BdBM20]{BdBM:DSGSSBPQT}
Pedro Boavida~de Brito and Ieke Moerdijk.
\newblock Dendroidal spaces, {$\Gamma$}-spaces and the special
  {B}arratt-{P}riddy-{Q}uillen theorem.
\newblock {\em Journal f{\"u}r die reine und angewandte Mathematik (Crelles
  Journal)}, 2020(760):229--265, 2020.

\bibitem[BHS22]{BHS:EAP}
Shaul Barkan, Rune Haugseng, and Jan Steinebrunner.
\newblock Envelopes for algebraic patterns.
\newblock {\em arXiv:2208.07183}, 2022.

\bibitem[Bri17]{B:SOGC}
Pedro Boavida~de Brito.
\newblock Segal objects and the {Grothendieck} construction, November 2017.
\newblock arXiv:1605.00706.

\bibitem[BV06]{BV:HIASTS}
John~M. Boardman and Rainer~M. Vogt.
\newblock {\em Homotopy invariant algebraic structures on topological spaces},
  volume 347.
\newblock Springer, 2006.

\bibitem[CHH18]{CHH:TMHTIO}
Hongyi Chu, Rune Haugseng, and Gijs Heuts.
\newblock Two models for the homotopy theory of $\infty$-operads.
\newblock {\em Journal of Topology}, 11(4):857--873, December 2018.

\bibitem[Cis19]{C:HCHA}
Denis-Charles Cisinski.
\newblock {\em Higher categories and homotopical algebra}, volume 180.
\newblock Cambridge University Press, 2019.

\bibitem[CM11]{CM:DSMHO}
Denis-Charles Cisinski and Ieke Moerdijk.
\newblock Dendroidal sets as models for homotopy operads.
\newblock {\em Journal of Topology}, 4(2):257--299, 2011.

\bibitem[CM13]{CM:DSSIO}
Denis-Charles Cisinski and Ieke Moerdijk.
\newblock Dendroidal segal spaces and $\infty$-operads.
\newblock {\em Journal of Topology}, 6(3):675--704, 2013.

\bibitem[Gro20]{G:SCIO}
Moritz Groth.
\newblock A short course on $\infty$-categories.
\newblock In {\em Handbook of homotopy theory}, pages 549--617. Chapman and
  Hall/CRC, 2020.

\bibitem[Hau22]{H:IOVSS}
Rune Haugseng.
\newblock $\infty$-operads via symmetric sequences.
\newblock {\em Mathematische Zeitschrift}, 301(1):115--171, 2022.

\bibitem[Heu11]{He:AOIO}
Gijs Heuts.
\newblock Algebras over $\infty$-operads.
\newblock {\em arXiv:1110.1776}, 2011.

\bibitem[HHM16]{HeuHiMoe:OEBLMDMIO}
Gijs Heuts, Vladimir Hinich, and Ieke Moerdijk.
\newblock On the equivalence between {Lurie}'s model and the dendroidal model
  for infinity-operads.
\newblock {\em Advances in Mathematics}, 302:869--1043, 2016.

\bibitem[Hin15]{H:RAM}
Vladimir Hinich.
\newblock Rectification of algebras and modules.
\newblock {\em Documenta Mathematica}, 20(2015):879--926, 2015.

\bibitem[HK24]{HK:IOSMIC}
Rune Haugseng and Joachim Kock.
\newblock {$\infty$-operads as symmetric monoidal $\infty$-categories}.
\newblock {\em Publicacions Matemàtiques}, 68(1):111 -- 137, 2024.

\bibitem[HM22]{HeMo:SDHT}
Gijs Heuts and Ieke Moerdijk.
\newblock {\em Simplicial and dendroidal homotopy theory}, volume~75 of {\em
  Ergeb. Math. Grenzgeb., 3. Folge}.
\newblock Cham: Springer, 2022.

\bibitem[HM24]{HM:OELIODIO}
Vladimir Hinich and Ieke Moerdijk.
\newblock On the equivalence of lurie's $\infty$-operads and dendroidal
  $\infty$-operads.
\newblock {\em Journal of Topology}, 17(4):e70003, 2024.

\bibitem[Ker23]{K:MEGCDSO}
David Kern.
\newblock Monoidal envelopes and {G}rothendieck construction for dendroidal
  {S}egal objects.
\newblock {\em arXiv:2301.10751}, 2023.

\bibitem[KK24]{KK:IOFEC}
Manuel Krannich and Alexander Kupers.
\newblock $\infty$-operadic foundations for embedding calculus.
\newblock {\em arXiv:2409.10991}, 2024.

\bibitem[Lur09a]{Lu:HA}
Jacob Lurie.
\newblock {\em Higher Algebra}.
\newblock Academic Search Complete. Princeton University Press, 2009.

\bibitem[Lur09b]{Lu:HTT}
Jacob Lurie.
\newblock {\em Higher topos theory}.
\newblock Princeton University Press, 2009.

\bibitem[May06]{M:TGILS}
J~Peter May.
\newblock {\em The geometry of iterated loop spaces}, volume 271.
\newblock Springer, 2006.

\bibitem[MW07]{MW:DS}
Ieke Moerdijk and Ittay Weiss.
\newblock Dendroidal sets.
\newblock {\em Algebr. Geom. Topol.}, 7:1441--1470, 2007.

\bibitem[Pra25]{P:RDLF}
Francesca Pratali.
\newblock Rectification of dendroidal left fibrations.
\newblock {\em arXiv:2502.17415}, 2025.

\bibitem[Ram22]{R:MGCIC}
Maxime Ramzi.
\newblock A monoidal {G}rothendieck construction for $\infty$-categories.
\newblock {\em arXiv:2209.12569}, 2022.

\bibitem[Rez01]{R:AMHTHT}
Charles Rezk.
\newblock A model for the homotopy theory of homotopy theory.
\newblock {\em Transactions of the American Mathematical Society},
  353(3):973--1007, 2001.

\end{thebibliography}
\end{document}